\documentclass[11pt,a4paper]{amsart}

\usepackage{amsfonts,amsmath,amssymb,amsthm}

\usepackage[utf8]{inputenc}
\usepackage[T1]{fontenc}

\usepackage{float} 
\usepackage{hyperref} 
\usepackage{stmaryrd}
\usepackage{siunitx}

\newcommand{\Zz}{\mathbb{Z}}
\newcommand{\Qq}{\mathbb{Q}}

\renewcommand {\geq}{\geqslant}
\renewcommand {\le}{\leqslant}
\renewcommand {\ge}{\geqslant}
\renewcommand {\epsilon}{\varepsilon}

\newcommand{\notdivides}{\mathrel{\not|}}

\newcommand{\defi}[1]{\emph{\textbf{#1}}}
\newcommand{\ul}[1]{\mathbf{#1}}

\newcommand{\dd}{\mathrm{d}}

\newcommand{\Tree}{\mathop{\mathrm{Tree}}\nolimits}
\newcommand{\Trunk}{\mathop{\mathrm{Trunk}}\nolimits}
\newcommand{\Fan}{\mathop{\mathrm{Fan}}\nolimits}
\newcommand{\val}{\mathop{\mathrm{val}}\nolimits}
\newcommand{\mult}{\mathop{\mathrm{mult}}\nolimits}
\newcommand{\Poinc}{\mathop{\mathrm{Poinc}}\nolimits}

\newcommand{\reduc}[1]{\xoverline{#1}}

\newcommand{\ZZbb}{\mathbf{Z}}

\makeatletter
\renewcommand{\pod}[1]{\allowbreak\mathchoice
	{\if@display \mkern 8mu\else \mkern 8mu\fi (#1)}
	{\if@display \mkern 8mu\else \mkern 8mu\fi (#1)}
	{\mkern4mu(#1)}
	{\mkern4mu(#1)}
}
\makeatother

\makeatletter
\newsavebox\myboxA
\newsavebox\myboxB
\newlength\mylenA

\newcommand*\xoverline[2][0.75]{%
	\sbox{\myboxA}{$\m@th#2$}%
	\setbox\myboxB\null
	\ht\myboxB=\ht\myboxA%
	\dp\myboxB=\dp\myboxA%
	\wd\myboxB=#1\wd\myboxA
	\sbox\myboxB{$\m@th\overline{\copy\myboxB}$}
	\setlength\mylenA{\the\wd\myboxA}
	\addtolength\mylenA{-\the\wd\myboxB}%
	\ifdim\wd\myboxB<\wd\myboxA%
	\rlap{\hskip 0.5\mylenA\usebox\myboxB}{\usebox\myboxA}%
	\else
	\hskip -0.5\mylenA\rlap{\usebox\myboxA}{\hskip 0.5\mylenA\usebox\myboxB}%
	\fi}
\makeatother

\newcommand{\sauteligne}{\leavevmode}

\makeatletter
\def\subsection{\@startsection{subsection}{2}%
	\z@{.5\linespacing\@plus.7\linespacing}{.3\linespacing}%
	{\normalfont\bfseries}}
\makeatother

{\theoremstyle{plain}
	\newtheorem{theorem}{Theorem}[section]    
	\newtheorem{lemma}[theorem]{Lemma}       
	\newtheorem{proposition}[theorem]{Proposition}      
	\newtheorem{corollary}[theorem]{Corollary}      
	\newtheorem*{theorem*}{Theorem}
}
{\theoremstyle{remark}
	\newtheorem{definition}[theorem]{Definition}      
	\newtheorem{remark}[theorem]{Remark}   
	
}

\usepackage{tikz} 
\usetikzlibrary{calc}
\usetikzlibrary{decorations.pathreplacing}

\newcommand{\myfigure}[2]{
	\begin{center}\small
		\tikzstyle{every picture}=[scale=1.0*#1]
		#2
\end{center}}

\newcommand*\circledtextH{%
	\tikz[baseline=(C.base)]\node[draw,circle,inner sep=1pt](C) {H};\!
}

\usepackage[a4paper]{geometry}
\title[Polynomial equations modulo a prime power]
{Solutions of polynomial equations\\ in several variables\\ modulo a prime power}

\author{Arnaud Bodin}
\author{Christian Drouin}

\email{arnaud.bodin@univ-lille.fr}
\email{christian.drouin@wanadoo.fr}

\address{Université de Lille, CNRS, Laboratoire Paul Painlevé, 59000 Lille, France}
\address{Seignosse, France}

\subjclass[2020] {Primary 11A07; Sec. 11D45, 11S05}

\keywords{Polynomial, Congruence, Tree, Poincaré series}

\date{\today}

\begin{document}

\begin{abstract}
We explain how to obtain the set of solutions of a multivariate polynomial equation modulo a power of a prime number. These solutions are determined by a tree, called the \emph{trunk}, which makes it possible to reconstruct all solutions. We apply these methods to determine the number of solutions, without having to enumerate them. We also illustrate these techniques by proving a simple case of Igusa’s theorem: the Poincaré series associated with a polynomial in two separated variables is rational.
\end{abstract}

\maketitle


\section{Introduction}

\subsection{Objectives}

The main objective is to explain how to solve polynomial equations modulo an integer~$n$:
\begin{equation}
	P(x_1,\ldots,x_n) \equiv 0 \pmod{n}.
	\label{eq:modn}                  
\end{equation}
The first step consists in applying the Chinese remainder theorem in order to reduce to the case where the modulus is a power of a prime number. Indeed, if $n=p_1^{e_1} \cdots p_r^{e_r}$, then the solutions of~\eqref{eq:modn} are the solutions of
\begin{equation}
	\label{eq:modpiei}
	P(x_1,\ldots,x_n) \equiv 0 \pmod{p_i^{e_i}} \qquad \text{for all } 1 \le i \le r.
\end{equation}

In the remainder of the paper, we will consider only one of these equations, where $p$ is a prime number and $e \ge 1$:
\begin{equation}
	\label{eq:modpe}
	P(x_1,\ldots,x_n) \equiv 0 \pmod{p^e}
\end{equation}

For $e=1$, $\Zz/p\Zz$ is a field; in one variable, the number of solutions of~\eqref{eq:modpe} is bounded by the degree~$d$ of the polynomial, and in several variables by $dp^{n-1}$ via the Schwartz--Zippel lemma.
The main difficulty we will encounter is that for $e\ge2$, $\Zz/p^e\Zz$ is not an integral domain, and therefore, even in one variable, the number of solutions can be much larger than the degree.

\medskip

The one-variable case, based on Schmidt and Stewart~\cite{ScSt}, was studied in detail by the authors in~\cite{BD26}: it is explained there how to obtain and count all solutions of an equation $P(x) \equiv 0 \pmod{p^e}$ using a \emph{trunk}, which makes it possible to compute the \emph{solution tree}. One will also find in that article many examples in one variable, useful to understand before studying the multivariate case.

\subsection{Solutions modulo $p^e$}

It is surprising that these notions generalize without difficulty to equations in several variables, of course with some adaptations.
We will thus define the notions of \emph{trunk}, \emph{thickness}, and \emph{tree-top function} (with respective notations $\Trunk$, $t$, $\varphi$).
All these notions are defined in Section~\ref{sec:trunk}.
This will allow us to make all solutions explicit; see Theorem~\ref{th:trunktotree}.

If one only wishes to count the solutions, without enumerating them one by one, here is a formula (Theorem~\ref{th:count}), which will be proved in Section~\ref{sec:count}.
\begin{theorem}
	\label{th:count-intro}
	Let $P \in \Zz[x_1,\ldots,x_n]$.
	The number $N_e$ of solutions of the equation $P(x_1,\ldots,x_n) \equiv 0 \pmod{p^e}$ in $(\Zz/ p^e \Zz)^n$ is:
	\[
	N_e = \sum_{\substack{(\ul{r},k) \in \Trunk(P) \\ \varphi(\ul{r},k) -t_k \, < \, e \, \le \, \varphi(\ul{r},k)}} 
	p^{n(e-k)}.	
	\]
\end{theorem}

\subsection{The Poincaré series}

Recall that $N_e$ denotes the number of solutions of the equation
$P(x_1,\ldots,x_n) \equiv 0 \pmod{p^e}$.
The \emph{Poincaré series} is the generating series:
\[
S(T) = \sum_{e \geq 0} N_e \frac{T^e}{p^{ne}}
\]
(with $N_0=1$).

The celebrated theorem of Igusa~\cite{igusa} asserts that this series is in fact a rational fraction in~$T$.
It was originally a conjecture of Borewicz--Šafarevič~\cite{borevich-shafarevich}.
Another proof was given by Denef (see for instance~\cite{denef}, see also the survey~\cite{PoVe}).

A recent work~\cite{chakrabarti-saxena} establishes Igusa’s theorem in a very general framework, relying on a powerful algorithmic approach.
See also the thesis~\cite{chakrabarti}.
In the present article, we deliberately focus on the special case of two variables, and more precisely on the situation of a separated polynomial of the form $F(x)+G(y)$, for which
we present an explicit computation of the rational fraction.

In the case of a univariate polynomial,
the detailed study of a tree, as well as the algorithm we use,
which we call the \emph{trunk algorithm} or \emph{thickness algorithm},
appears in Schmidt and Stewart~\cite{ScSt},
in Zúñiga-Galindo~\cite{ZG2003},
in Berthomieu, Lecerf and Quintin~\cite{BLQ2013},
and very explicitly in Kopp, Randall, Rojas and Zhu~\cite{Kopp},
as well as in Dwivedi, Mittal and Saxena~\cite{DMS2019,DMS2019P4,DS2020}.

Computations of the Poincaré function in special cases
have been treated by different methods, for example, among others,
by~\cite{hayes-nutt} and~\cite{bollaerts}.

Most of the existing studies are technical and difficult.
We propose to illustrate the techniques developed here in order to give a simple proof of this result of Igusa in the special case of two separated variables.

\begin{theorem}
	\label{th:igusa-intro}
	Let $p$ be a prime number.
	Let $P(x,y) = F(x)+G(y) \in \Zz[x,y]$ be a separated polynomial in two variables ($n=2$).	
	Then the Poincaré series is a rational fraction:
	\[
	S(T) = \sum_{e \geq 0} N_e \frac{T^e}{p^{2e}} \quad \in\Qq[T].
	\]
\end{theorem}

Moreover, the expression as a rational fraction could be computed from the trunk; here we will do so in practice only for certain special cases.
Everything concerning Poincaré series and the proof of this result is contained in Section~\ref{sec:count} and the following sections.

In the appendix, we provide the detailed example of $P(x,y) = x^2-y^3$ and of its solutions modulo~$p$.
One finds there the computation of the trunk, of the Poincaré series, and of the number of solutions of
$x^2-y^3 \equiv 0 \pmod{p^e}$.

\section{The trunk}
\label{sec:trunk}

Fix a prime integer~$p$.
Let $\ZZbb$ denote either the ring $\Zz$ of integers or the ring $\Zz_p$ of $p$-adic integers.
Let $P(x_1,\ldots,x_n)$ be a polynomial in $n \ge 1$ variables with coefficients in $\ZZbb$.
We write $\ul{x} = (x_1,\ldots,x_n)$.

\subsection{Thickness}

\begin{definition}
	\label{def:thickness}
	The \defi{thickness} of $P(\ul{x}) \in \ZZbb[\ul{x}]$ at the point $\ul{x}_0 \in \Zz^n$ is the largest integer $t$ such that there exists $Q(\ul{x}) \in \ZZbb[\ul{x}]$ satisfying:
	\[
	P(\ul{x}_0 + p\ul{x} ) = p^t Q(\ul{x}).
	\]
	The polynomial $Q(\ul{x})$ is then called the \defi{successor} of $P$ at $\ul{x}_0$.
\end{definition}

The thickness $t$ is positive if and only if $\ul{x}_0 \in \Zz^n$ is a solution of
$P(\ul{x}) \equiv 0 \pmod{p}$.

\medskip

Throughout the remainder of the article, we shall assume that $P(\ul{x})$ is $p$-primitive
(i.e., not all coefficients are divisible by~$p$).
This can be done without loss of generality: indeed, if $P(\ul{x})$ were not $p$-primitive, we would write
$P(\ul{x})=p^{t_0}P_0(\ul{x})$ where $P_0$ is $p$-primitive.
Anticipating what follows, one would then assign the thickness $t_0$ to the vertex $(\ul{0},0)$ of the trunk,
and add the value $t_0$ to the definition of the tree-top function in order to ensure the validity of the statements.

\subsection{Solution tree}

\begin{definition}
	\label{def:omega}
	The \defi{$p$-adic congruence tree} $\Omega_p$ is given by the \emph{vertices}:
	\[
	\Omega_p = \Big\lbrace
	(\ul{x}, e) \mid \ul{x} \in (\Zz/p^e\Zz)^n, \; e \ge 0 \Big\rbrace.
	\]
	Two vertices $(\ul{x}, e)$ and $(\ul{x}', e+1)$ are connected by an \emph{edge}
	if and only if $\ul{x}' \equiv \ul{x} \pmod {p^e}$, that is,
	$x_i' \equiv x_i \pmod{p^e}$ for all $i=1,\ldots,n$.
\end{definition}

In this way, we endow $\Omega_p$ with a tree structure. The root is the vertex $(\ul 0, 0)$.
From each vertex emanate exactly $p^n$ outgoing edges.

\begin{remark}
	We sometimes use the natural identification of $\Zz/m\Zz$ with $\llbracket 0,m-1\rrbracket$.
	We denote by $k \% m$ the remainder of the Euclidean division of $k$ by $m$.
	
	In the one-variable case $(n=1)$, a finite path in the $p$-adic congruence tree $\Omega_p$ starting from the root corresponds to the $p$-adic expansion of an integer $N \in \Zz$ associated with the terminal vertex:
	\[
	N = a_0 + a_1p+\cdots +a_k p ^k
	\]
	where $a_i \in \llbracket 0, p-1 \rrbracket$.
	The $p$ possible values for $a_i$ correspond to the $p$ outgoing edges from a vertex.
	An infinite path in $\Omega_p$ corresponds to a $p$-adic integer $N \in \Zz_p$ via its expansion:
	\[
	N = a_0 + a_1p+\cdots +a_k p ^k +\cdots 
	\]
	
	In the case of $n$ variables, one works componentwise: a finite path in $\Omega_p$ corresponds to $(N_1,\ldots,N_n) \in \Zz^n$, and an infinite path to $(N_1,\ldots,N_n) \in \Zz_p^n$.
\end{remark}

\begin{definition}
	\label{def:tree}
	The \defi{solution tree} is defined by
	\[
	\Tree(P) = \Big\{
	(\ul{x}, e) \in \Omega_p \mid P(\ul{x}) \equiv 0 \pmod{p^e}
	\Big\}.
	\]
\end{definition}

It is a (finite or infinite) subtree of $\Omega_p$ containing the root.

We again refer to the exposition~\cite{BD26} for further explanations, examples, and figures.

\subsection{Trunk}

The trunk is a tree of $\Omega_p$ that makes it possible to reconstruct the solution tree $\Tree(P)$.
It is in fact a subtree of $\Tree(P)$, which may also be finite or infinite, but is simpler in the sense that it contains far fewer vertices.

\begin{definition}
	\label{def:trunk}
	The \defi{trunk} of $P(\ul{x}) \in \ZZbb[\ul{x}]$, denoted by $\Trunk(P)$, is defined recursively as follows:
	\begin{itemize}
		\item \emph{Initialization.}
		The root $(\ul{0}, 0)$ is an element of the trunk (this vertex has no thickness). The polynomial attached to this vertex is $P(\ul{x})$.
		
		\item \emph{Heredity.}
		Suppose that $(\ul{r}, k)$ is a vertex of the trunk attached to a polynomial $P_k(\ul{x})$.
		We construct the vertices of height $k+1$ attached to this vertex.
		Consider the equation $P_k(\ul{x}) \equiv 0 \pmod{p}$. For each
		of its solutions $\ul{r}_k \in (\Zz/p\Zz)^n$, we compute the decomposition associated with the thickness
		$P_k(\ul{r}_k + p \ul{x}) = p^{t_{k+1}} P_{k+1}(\ul{x})$.
		Then $(\ul{r}+ p^{k} \ul{r}_k, k+1)$ is a new element of the trunk.
		We associate to it the thickness $t_{k+1}$ and the polynomial $P_{k+1}$.
	\end{itemize}
\end{definition}

Thus, a vertex of the trunk at height $k$ is defined by an element
$\ul{r} = \ul{r}_0+ p\ul{r}_1+\cdots+p^{k-1}\ul{r}_{k-1}$,
where $\ul{r}_i\in \llbracket 0,p-1\rrbracket^n$ ($1\le i\le k-1$).
The outgoing edges are associated with the solutions $\ul{r}_k$ of
$P_k(\ul{x}) \equiv 0 \pmod{p}$.
In what follows, we will essentially consider the trunk as a tree weighted by a thickness at each vertex.

\subsection{Tree-top function}

\begin{definition}
	\label{def:treetop}
	Let $(\ul{r}, k) \in \Trunk(P)$.
	The \defi{tree-top function} at this vertex is defined by:
	\[
	\varphi(\ul{r},k) = t_1 + \cdots +t_k
	\]
	where $t_1,\ldots, t_k$ are the thicknesses of the vertices along the path connecting the vertex $(\ul{r}, k)$ to the root $(\ul{0}, 0)$.
\end{definition}

\begin{lemma}
	\label{lem:toptree}
	There exists a decomposition:
	\[
	P(\ul{r}+p^k \ul{x}) = p^{\varphi(\ul{r},k)} Q(\ul{x})
	\]
	where $Q(\ul{x}) \in \ZZbb[\ul{x}]$.
\end{lemma}

\begin{proof}
	Let $\ul r = \ul{r}_0+ p\ul{r}_1+\cdots+p^{k-1}\ul{r}_{k-1}$ be the $p$-adic decomposition
	of $\ul{r}$, obtained componentwise for each of the $n$ coordinates.
	We have $P(\ul{r}_0+p\ul{x}) = p^{t_1} P_1(\ul{x})$, $P_1(\ul{r}_1+p\ul{x}) = p^{t_2} P_2(\ul{x})$, hence
	$P(\ul{r}_0+p\ul{r}_1 +p^2\ul{x}) = p^{t_1+t_2}P_2(\ul{x})$, and the stated formula follows by induction.
\end{proof}

\subsection{From the trunk to the tree}

By definition:
\[
(\ul{x}_0,e) \in \Tree(P) 
\iff
P(\ul{x}_0) \equiv 0 \pmod{p^e} .
\]

Here is how to compute all solutions from the trunk; this is a multivariate generalization of~\cite{BD26}:
\begin{theorem}
	\label{th:trunktotree}
	\[
	P(\ul{x}_0) \equiv 0 \pmod{p^e} 
	\iff
	\text{ there exists }  (\ul{r},k) \in \Trunk(P) \text{ such that } 
	\left\lbrace\begin{array}{l}
		\ul{x}_0 \equiv \ul{r} \pmod{p^k}  \\
		\text{and } \varphi(\ul{r},k) \ge e
	\end{array}\right.
	\]
	
	Moreover, for a given solution $\ul{x}_0$, such a pair $(\ul{r},k)$ is unique if we impose the condition:
	\[
	\varphi(\ul{r},k) -t_k < e \le \varphi(\ul{r},k)
	\]
	where $t_k$ is the thickness of the vertex $(\ul{r},k)$.  	
\end{theorem}

Let us reformulate the first part of the theorem in graph-theoretic terms.
We say that the \defi{fan} from a vertex $(\ul{r},k)$ up to height $h$
is the set of vertices of $\Omega_p$ issued from the vertex $(\ul{r},k)$ up to height $h$:
\[
\Fan_{\le h}(\ul{r},k) = \big\{ (\ul{x},l) \in \Omega_p \mid  \ul{x} \equiv \ul{r} \pmod{p^k} \text{ and } l \le h \big\}.
\]
In graph-theoretic terms, Theorem~\ref{th:trunktotree} becomes:
\[
\Tree(P) = \bigcup_{(\ul{r},k) \in \Trunk(P)} \Fan_{\le  \varphi(\ul{r},k)}(\ul{r},k).
\]
The uniqueness will in turn be fundamental when we want to count the number of solutions; see Theorem~\ref{th:count}.

\begin{proof}
	Let $(\ul{x}_0, e)$ be an element of the $p$-adic congruence tree $\Omega_p$.
	Let $(\ul{r}, k) \in \Trunk(P)$ denote the most recent ancestor of $(\ul{x}_0, e)$ belonging to the trunk of~$P$.
	The relation between $\ul{x}_0$ and $\ul{r}$ is given by
	$\ul{x}_0 = \ul{r} + p^{k} \ul{y}_0$ (where $\ul{y}_0 \in\Zz^n$).
	
	By Lemma~\ref{lem:toptree}, since $(\ul{r}, k)$ is an element of the trunk, we know that there exists a decomposition:
	\[
	P(\ul{r} + p^{k} \ul{x}) = p^{\varphi(\ul{r}, k)} Q(\ul{x})
	\]
	where $\varphi(\ul{r},k)$ is the value of the tree-top function.
	
	The assumption that $(\ul{r}, k) \in \Trunk(P)$ is the most recent ancestor in the trunk means exactly that
	$Q(\ul{y}_0) \not\equiv 0 \pmod{p}$.
	
	We now have all the ingredients to conclude the proof:
	\begin{align*}
		P(\ul{x}_0) \equiv 0 \pmod{p^e}
		&\iff P(\ul{r}+p^{k}\ul{y}_0) \equiv 0 \pmod{p^e} \\
		&\iff p^{\varphi(\ul{r},k)} Q(\ul{y}_0) \equiv 0 \pmod{p^e} \\  	
		&\iff  \varphi(\ul{r},k) \ge e
	\end{align*}
	The assumption $Q(\ul{y}_0) \not\equiv 0 \pmod{p}$ is used in the last equivalence.
	
	It remains to prove the uniqueness characterization.
	For the moment, from a vertex $(\ul{r}, k)$ of the trunk, we have constructed a fan of solutions up to height $\varphi(\ul{r}, k)$.
	
	But from $(\ul{r}^-,k-1)$ (the direct ancestor of $(\ul{r},k)$), we construct solutions only up to height
	$\varphi(\ul{r}^-, k-1) = \varphi(\ul{r}, k) - t_k$.
	Thus the solutions whose height $e$ satisfies
	$\varphi(\ul{r}, k) - t_k < e \le \varphi(\ul{r}, k)$
	are determined solely by the vertex $(\ul{r}, k)$ of the trunk.
\end{proof}

\section{Properties of the thickness}

\subsection{Computation of the thickness}

Let $P(\ul{x}) \in \ZZbb[\ul{x}]$. Recall that the \defi{thickness} of $P$ at $\ul{r}$ is the largest integer $t$ such that one has a decomposition
$P(\ul{r}+p\ul{x}) = p^t Q(\ul{x})$.
We will use Taylor's formula to compute the thickness.
Indeed, the Taylor expansion applied to $P$ at $\ul{r}$ gives the decomposition:
\[
P(\ul{r} + \ul{x}) = \sum_{h=0}^{d} P_h(\ul{x}),
\]
where $d$ is the degree of $P$ and $P_h(\ul{x})$ is a homogeneous polynomial of degree $h$:
\[
P_h(\ul{x}) 
= \frac{1}{h!} \dd^{h} P(\ul{r})[\ul{x},\ldots,\ul{x}]
= \frac{1}{h!} \sum_{|\ul{i}|=h}  \partial^{\ul{i}} P (\ul{r}) \ul{x}^{\ul{i}}.
\]
We use the usual conventions for multi-indices:
$\ul{i} = (i_1,\ldots,i_n)$, $|\ul{i}| = i_1+\cdots+i_n$,
$\ul{x}^{\ul{i}} = x_1^{i_1} \cdots x_n^{i_n}$, and
$\partial^{\ul{i}} P (\ul{x}) = \frac{\partial^{i_1+\cdots+i_n} P}{\partial x_1^{i_1}\cdots\partial x_n^{i_n}}$.

\begin{lemma}
	\label{lem:taylor}
	
	\[
	t = 
	\min_{|\ul{i}|\ge0} \val_p \left( \frac{1}{|\ul{i}|!}{\partial^{\ul{i}} P(\ul{r})} \, p^{|\ul{i}|}\right)
	\]
	In particular $t \le \deg P$.
\end{lemma}

\begin{proof}
	
	The Taylor formula applied to $P(\ul{r} + p\ul{x})$ gives:
	\[
	P(\ul{r} + p\ul{x}) = \sum_{h=0}^{d} P_h(\ul{x}) p^h.
	\]
	
	Let $t$ denote the thickness and $t'$ the minimum appearing in the statement.
	Since $p^t$ divides $P_h(\ul{x}) p^h$ for all $0 \le h \le d$,
	we have $t \le t'$.
	Conversely, $p^{t'}$ divides all the coefficients of all the polynomials $P_h(\ul{x}) p^h$, hence $p^{t'}$ divides $P(\ul{r} + p\ul{x})$, so $t \ge t'$.
\end{proof}

\subsection{The thickness decreases along the trunk}

With the notation of the decomposition $P(\ul{r}_0+p\ul{x}) = p^t Q(\ul{x})$,
we define the \defi{residual degree} of $P$ at $\ul{r}_0$ as the integer
$s = \deg \reduc{Q}$, where $\reduc{Q}(\ul{x})$ denotes the reduction modulo $p$ of $Q(\ul{x})$.

We denote by $\mult(\reduc{P}, \ul{r}_0)$ the multiplicity of $\ul{r}_0$ as a root of $\reduc{P}$ in $(\Zz/p\Zz)^n$.
If $P(\ul{r}_0 + \ul{x}) = \sum_{h=0}^{d} P_h(\ul{x})$ is the decomposition into a sum of homogeneous polynomials,
then this multiplicity is the smallest integer $h$ such that
$\reduc{P_h}(\ul{x}) \not\equiv 0 \pmod{p}$.

Now consider two successive decompositions:
$P(\ul{r}_0+p\ul{x}) = p^t Q(\ul{x})$, followed by
$Q(\ul{r}_1+p\ul{x}) = p^{t'} R(\ul{x})$.

\begin{lemma}
	\label{lem:inequalities}
	
	\[
	\mult(\reduc{P}, \ul{r}_0) 
	\ge t
	\ge s
	\ge \mult(\reduc{Q}, \ul{r}_1) 
	\]
	
	In particular, $t \ge t'$.
\end{lemma}

An important consequence is that since $t \ge t'$, the thickness decreases as one moves upward along the trunk.
Thus, along an infinite branch, the thickness eventually becomes constant.

\begin{proof}
	\sauteligne
	\begin{itemize}
		
		\item Let $P(\ul{r}_0 + \ul{x}) = \sum_{h=0}^{d} P_h(\ul{x})$ be the decomposition into a sum of homogeneous polynomials.
		If the thickness of $P$ at $\ul{r}_0$ is $t$, then
		$p^t$ divides $P(\ul{r}_0 + p\ul{x})$, hence divides
		$p^h P_h(\ul{x})$ for all $h$.
		In particular, for $h<t$, $p$ divides $P_h(\ul{x})$.
		Therefore $\mult(\reduc{P},\ul{r}_0) \ge t$.
		
		\item Since $P(\ul{r}_0 + \ul{x}) = p^t Q(\ul{x})$, we have
		$Q(\ul{x}) = \sum_{h=0}^{d} p^{h-t} P_h(\ul{x})$.
		Thus for $h > t$, $p$ divides $p^{h-t} P_h(\ul{x})$.
		Hence $\deg \reduc{Q}(\ul{x}) \le t$, so $s \le t$.
		
		\item We always have $\mult(\reduc{Q}, \ul{r}_1) \le \deg(\reduc{Q}) = s$.
		
		\item Finally, by the very first inequality of this lemma, applied to $Q(\ul{x})$, we have
		$\mult(\reduc{Q}, \ul{r}_1) \ge t'$.
		Hence in particular $t \ge t'$.
	\end{itemize}
\end{proof}

\subsection{Hensel's lemma}
\label{ssec:hensel}

We have already observed that a polynomial has positive thickness at $\ul{r}_0$ if and only if
$P(\ul{r}_0) \equiv 0 \pmod{p}$.
We now study the case of thickness $t=1$.
According to Lemma~\ref{lem:taylor}, if $t=1$ then:
\begin{itemize}
	\item either $p \mid P(\ul{r}_0)$ and $p^2 \notdivides P(\ul{r}_0)$,
	\item or $p \mid P(\ul{r}_0)$ and there exists $1 \le i_0 \le n$ such that
	$p \notdivides \frac{\partial P}{\partial x_{i_0}}(\ul{r}_0)$.
\end{itemize}
Both cases may occur simultaneously.

\begin{lemma}
	\label{lem:onefinite} 
	If $t=1$ and all partial derivatives vanish:
	$\frac{\partial P}{\partial x_{i}}(\ul{r}_0) \equiv 0 \pmod{p}$ for $1\le i \le n$,
	then the trunk stops immediately after this vertex of thickness~$1$.
\end{lemma}

\begin{proof}
	By assumption,
	$P(\ul{r}_0+\ul{x}) = c_0 + \text{terms of degree $\ge 2$}$, hence
	\[
	P(\ul{r}_0+p\ul{x}) = c_0 + p^2 R(\ul{x}).
	\]
	We assumed $t=1$, so $p \mid c_0$ and $p^2 \notdivides c_0$.
	Thus
	\[
	P(\ul{r}_0+p\ul{x}) = p\left( \frac{c_0}{p} + p R(\ul{x}) \right).
	\]
	Therefore the successor of $P$ is
	$Q(\ul{x}) = \frac{c_0}{p} + p R(\ul{x})$.
	Since $\frac{c_0}{p}$ is not divisible by $p$, the congruence
	$\reduc{Q}(\ul{x}) \equiv 0 \pmod{p}$ has no solutions, and hence the trunk stops at this point.
\end{proof}

In the case of thickness~$1$, when one of the partial derivatives does not vanish at $\ul{r}_0$,
we say that we are in the \emph{Hensel lemma situation}.
In this case as well, we obtain a complete description of the trunk.

\begin{definition}
	\label{def:henseltree}
	A \defi{Hensel tree} is an infinite tree in which each vertex admits exactly $p^{n-1}$ outgoing edges and where all thicknesses are equal to~$1$.
\end{definition}

\begin{lemma}
	\label{lem:hensel}
	If $P(\ul{r}_0) \equiv 0 \pmod{p}$ and there exists $1 \le i \le n$ such that
	$\frac{\partial P}{\partial x_{i}}(\ul{r}_0) \not\equiv 0 \pmod{p}$,
	then the part of the trunk of $P$ issued from $\ul{r}_0$ is a Hensel tree.
\end{lemma}

\begin{proof}
	Without loss of generality, assume that
	$\frac{\partial P}{\partial x_1}(\ul{r}_0) \not\equiv 0 \pmod{p}$.
	Write
	$P(\ul{r}_0+\ul{x}) = c_0 + c_1 x_1 + c_2 x_2+\cdots +c_n x_n + \cdots$.
	By Lemma~\ref{lem:taylor}, we know that $\ul{r}_0$ has thickness~$1$, which gives:
	\[
	P(\ul{r}_0+p\ul{x}) = p \left( \frac{c_0}{p} + c_1 x_1 + c_2 x_2+\cdots +c_n x_n + pR(\ul{x}) \right).
	\]
	The successor of $P$ is therefore
	$Q(\ul{x}) = \frac{c_0}{p} + c_1 x_1 + c_2 x_2+\cdots +c_n x_n + pR(\ul{x})$.
	Its reduction modulo $p$ is
	$\reduc{Q}(\ul{x}) = \frac{c_0}{p} + c_1 x_1 + c_2 x_2+\cdots +c_n x_n$,
	with $c_1 \not\equiv 0 \pmod{p}$ by assumption.
	Thus, for any choice of $x_2,\ldots,x_n \in \Zz/p\Zz$, there exists a unique
	$x_1 \in \Zz/p\Zz$ giving a solution to the linear equation
	$\frac{c_0}{p} + c_1 x_1 + c_2 x_2+\cdots +c_n x_n \equiv 0 \pmod{p}$.
	Hence $\reduc{Q}(\ul{x}) \equiv 0 \pmod{p}$ admits exactly $p^{n-1}$ solutions.
	
	Let $\ul{r}_1$ be such a solution.
	Since $\reduc{Q}(\ul{r}_1) \equiv 0 \pmod{p}$, the thickness is at least~$1$.
	And since
	$\frac{\partial Q}{\partial x_1}(\ul{r}_0) \equiv c_1 \not\equiv 0 \pmod{p}$,
	we have $t=1$ and we are again in the situation of Hensel’s lemma.
	The conclusion then follows by induction.
\end{proof}

\section{The number of solutions}
\label{sec:count}

\subsection{Formula}

We recall the theorem stated in the introduction. We now have all the ingredients to prove it.
\begin{theorem}
	\label{th:count}
	The number of solutions to the equation $P(\ul{x}) \equiv 0 \pmod{p^e}$ in $(\Zz/ p^e \Zz)^n$ is:
	\[
	N_e = \sum_{\substack{(\ul{r},k) \in \Trunk(P) \\ \varphi(\ul{r},k) -t_k \, < \, e \, \le \, \varphi(\ul{r},k)}} 
	p^{n(e-k)}.	
	\]
\end{theorem}

\begin{proof}
	Theorem~\ref{th:trunktotree} tells us that each vertex $(\ul{r}, k)$ of the trunk provides, in a unique way, the solutions $(\ul{x}_0, e)$ of
	$P(\ul{x}) \equiv 0 \pmod{p^e}$ if and only if
	$(\ul{x}_0, e)$ is a successor of $(\ul{r}, k)$ and
	$\varphi(\ul{r},k) -t_k  <  e  \le  \varphi(\ul{r},k)$.
	
	Therefore, the number of solutions arising from the fan of $(\ul{r}, k)$ with height exactly $e$ is $(p^{n})^{e-k}$.
	This yields the desired formula.
\end{proof}

\subsection{Generating function}

\begin{definition}
	\label{def:poincare}
	The \defi{Poincaré series} is:
	\[
	S(T) = \sum_{e \geq 0} N_e \frac{T^e}{p^{ne}},
	\]
	where $N_e$ denotes the number of solutions to the equation $P(\ul{x}) \equiv 0 \pmod{p^e}$
	(with the convention that $N_0 = 1$).
	We recall that $n$ is the number of variables.
\end{definition}

\subsection{Formula from the trunk}

We will use Theorem~\ref{th:count} in order to compute the Poincaré series directly from the trunk.

\begin{theorem}
	\label{th:xylon}
	\[
	S(T) = \sum_{(\ul{r},k) \in \Trunk(P)} \frac{T^{\varphi(\ul{r}, k) - t_k + 1}}{p^{nk}}\big( 1 + T + T^2 + \cdots + T^{t_k-1} \big)
	\]
\end{theorem}

We define
\[
\Xi_k(T) = \frac{T^{\varphi(\ul{r}, k) - t_k + 1}}{p^{nk}}\big( 1 + T + T^2 + \cdots + T^{t_k-1} \big).
\]
These functions of $T$ are called \defi{xylon functions}, and the proposition will be called
the \defi{xylon formula}:
\[
S(T) = \sum_{(\ul{r},k) \in \Trunk(P)} \Xi_k(T).
\]

Remarks. By convention, the contribution of the root of the trunk is $\Xi_0(T) = 1$.
On the other hand, if $t_k = 1$, then the sum $1 + T + T^2 + \cdots + T^{t_k-1}$ reduces to $1$.

\begin{proof}
	To simplify the notation, we write $t = t_k$ and $\varphi = \varphi(\ul{r}, k)$.
	We reorganize the sum $S(T)$ using Theorem~\ref{th:count}:
	\begin{align*}
		S(T) 
		&= \sum_{e \geq 0} N_e \frac{T^e}{p^{ne}} \\
		&= \sum_{e \geq 0} \quad  \sum_{\substack{(\ul{r},k) \in \Trunk(P) \\ \varphi-t <  e \le \varphi}} p^{n(e-k)} \frac{T^e}{p^{ne}} \\
		&= \sum_{(\ul{r},k) \in \Trunk(P)} \quad \sum_{\varphi-t < e \le \varphi} \frac{T^e}{p^{nk}} \\
		&= \sum_{(\ul{r},k) \in \Trunk(P)} \frac{1}{p^{nk}} \big( T^{\varphi-t+1} + T^{\varphi-t+2} + \cdots + T^{\varphi} \big)	\\	
		&= \sum_{(\ul{r},k) \in \Trunk(P)} \frac{T^{\varphi- t + 1}}{p^{nk}}\big( 1 + T + T^2 + \cdots + T^{t-1} \big)
	\end{align*}
\end{proof}

We now give an application of the xylon formula to two particular cases.

\subsection{Infinite stalk of constant thickness}

\begin{lemma}
	\label{lem:Sstalk}
	The generating series associated with an infinite stalk of constant thickness $t=u$ issued from the root is a rational fraction:
	\[
	S(T) = 1 \  + \ \big( 1+T+T^2+\cdots + T^{u-1} \big)\frac{T}{p^n} \left( 1-\frac{T^u}{p^n} \right)^{-1} \quad \in \Qq[T]
	\]
\end{lemma}

\begin{proof}
	By definition, at each height $k$ there is a unique vertex whose thickness equals $t=u$.
	Then the tree-top function at height $k$ equals $\varphi_k = k u$.
	Thus:
	\begin{align*}
		S(T) 
		&= \sum_{k \ge 0} \Xi_k(T) \\
		&=  1 \  + \ \sum_{k>0} \frac{T^{(k-1)u+1}}{p^{nk}} \big( 1+T+T^2+\cdots + T^{u-1} \big) \\
		&=  1 \  + \ \big( 1+T+T^2+\cdots + T^{u-1} \big)\sum_{k \ge 0}  \frac{T^{ku+1}}{p^{nk+n}}  \\     
		&=  1 \  + \ \big( 1+T+T^2+\cdots + T^{u-1} \big)\sum_{k \ge 0} \frac{T}{p^n}\left( \frac{T^u}{p^n} \right)^k \\
		&= 1 \  + \ \big( 1+T+T^2+\cdots + T^{u-1} \big)\frac{T}{p^n} \left( 1-\frac{T^u}{p^n} \right)^{-1}
	\end{align*}
\end{proof}

\subsection{The Hensel lemma case}

A Hensel tree is an infinite tree in which each vertex has exactly $p^{n-1}$ outgoing edges, all vertices having thickness~$1$.
\begin{lemma}
	\label{lem:Shensel}
	The generating series associated with a Hensel tree is a rational fraction:
	\[
	S(T) =\left( 1-\frac{T}{p} \right)^{-1} \ \in \Qq[T]
	\]
\end{lemma}

\begin{proof}
	For a given height $k$, there are $(p^{n-1})^k$ vertices. Each vertex of height $k$ has tree-top value $\varphi_k =k$ and
	its contribution to the sum is $\Xi_k(T) = \frac{T^k}{p^{nk}}$.
	Thus:
	\[
	S(T) 
	= \sum_{k>0} p^{(n-1)k} \Xi_k(T) 
	= 1+\sum_{k>0} p^{(n-1)k} \frac{T^k}{p^{nk}} 
	= \sum_{k\ge0}\left( \frac{T}{p} \right)^k 
	= \left( 1-\frac{T}{p} \right)^{-1}
	\]
\end{proof}

For a Hensel tree that would not be attached to the root but to a vertex of height $k_0$ and with tree-top value $\varphi_0$, one would similarly show:
\[
S_{k_0,\varphi_0}(T) = \frac{T^{\varphi_0}}{p^{nk_0}} \left(1-\frac{T}{p}\right)^{-1}.
\]

\section{Normalization}
\label{sec:normal}

The purpose of this section is to provide preparatory results and to reduce to the simplest possible form of the bivariate polynomials to be studied.
We will consider only \defi{non-degenerate} polynomials, that is, $\deg_x P \ge 1$ and $\deg_y P \ge 1$.

\subsection{The trunk is invariant under translation}

\begin{lemma}
	\label{lem:translation}
	Let $P(x,y) \in \Zz_p[x,y]$ and $\alpha,\beta \in\Zz_p$.
	Define $Q(x,y) \in \Zz_p[x,y]$ by $Q(x,y) = P(x-\alpha, y-\beta)$.
	Then the trunks of $P$ and $Q$ are isomorphic (i.e.~there exists a bijection between these trees that preserves the thicknesses).
\end{lemma}

\begin{corollary}
	\label{cor:translation}
	For an infinite stalk of a trunk, one may assume after translation that it lies above the roots $(0,0)$ at all heights.
\end{corollary}

\begin{proof}
	\emph{Isomorphism between the roots of $P$ and $Q$ modulo $p$.}
	Let $(r,s) \in (\Zz/p\Zz)^2$ be a root of $P$ modulo $p$:
	\[
	P(r,s) \equiv 0 \pmod{p} 
	\iff Q(r+\alpha, s+\beta) \equiv 0 \pmod{p} 
	\iff Q(r+\alpha_0, s+\beta_0) \equiv 0 \pmod{p} 
	\]
	where we write $\alpha = \alpha_0 + \alpha_1 p + \alpha_2 p^2+\cdots = \alpha_0  + p\alpha'$, and similarly for $\beta$.
	Thus:
	\[
	(r,s) \text{ is a root of } P(x,y) \text{ modulo } p
	\iff (r',s')=(r+\alpha_0, s+\beta_0) \text{ is a root of } Q(x,y) \text{ modulo } p
	\]	
	
	\medskip
	
	\emph{Preservation of thicknesses.}
	Let $(r,s)$ be a root of thickness $t$ for $P$:
	\begin{align*}
		& P(r+px, s+py) = p^t \tilde P(x,y) \\
		\iff\ & Q(r+\alpha+px, s+\beta+py) = p^t \tilde P(x,y)  \\
		\iff\ & Q(r+\alpha_0+p(x+\alpha'), s+\beta_0+p(y+\beta')) = p^t \tilde P(x,y)  \\		
		\iff\ & Q(r'+pX, s'+pY) = p^t \tilde P(X-\alpha',Y-\beta') \qquad \text{ with } X = x+\alpha', \; Y=y+\beta'  \\			
		\iff\ & Q(r'+pX, s'+pY) = p^t \tilde Q(X,Y)
	\end{align*}
	where we have defined $\tilde Q(X,Y) = \tilde P(X-\alpha',Y-\beta')$.
	Thus $Q$ also has thickness $t$ at the root $(r',s')$.
	
	\medskip
	
	\emph{Conclusion.}
	We have just shown that the successor $\tilde Q(X,Y) = \tilde P(X-\alpha', Y-\beta')$ is itself a translate of the successor $\tilde P(X,Y)$.
	By induction, moving up along the tree, this proves the result for all vertices.
\end{proof}

\subsection{Normalization}

We consider a polynomial $P(x,y)$ associated with an infinite stalk of constant thickness $\ge 2$.
We wish to find a simple expression for $P(x,y)$.
The guiding idea is the following: if a univariate polynomial $F(x)$ admits an infinite stalk of constant thickness equal to $u$ at the root $0$, then one can factor $F(x)$ in the form
$F(x) = x^u U(x)$.

\begin{lemma}
	\label{lem:normnew}
	\sauteligne
	\begin{enumerate}
		\item If $P(x,y) = F(x) + G(y)$ (with $p \notdivides P(x,y)$) admits an infinite stalk whose every vertex has thickness $\ge u$ at the root $(0,0)$, then, up to the symmetry $(x,y) \leftrightarrow (y,x)$,
		\[
		P(x,y) = x^u U(x) + p^\alpha y^v V(y) 
		\]
		with $\alpha \ge 0$, $v \ge u$, and $U(x)\in \Zz_p[x]$, $V(y) \in \Zz_p[y]$.
		
		\item Let $K \ge 0$. Let $U(x) = \sum_i u_i x^i$ and $V(y) = \sum_j v_j y^j$.
		By going sufficiently high on the stalk of the roots $(0,0)$, which we shall call the \emph{principal stalk}, one may assume that
		$U(x)\equiv u_0\pmod{p^K}$ and $V(y)\equiv v_0\pmod{p^K}$.
		
		\item If moreover the stalk has constant thickness $u$, then $U(0) \not\equiv 0 \pmod{p}$.
		At a sufficiently large level along the stalk, one can also place oneself in a situation where, in addition,
		$V(0) \not\equiv 0 \pmod{p}$ (by modifying the parameters $\alpha$, $v$, and the polynomial $V(y)$).
		
	\end{enumerate}	
	
\end{lemma}

Thus, thanks to Lemma~\ref{lem:translation}, if one had an infinite stalk whose limit vertex at infinity is $(\rho,\sigma) \in \Zz_p^2$, then the decomposition would be:
\[
P(x,y) = (x-\rho)^u U(x) + p^\alpha (y-\sigma)^v V(y).
\]

\begin{proof}
	~
	\begin{enumerate}
		\item 
		We first show that $P(0,0)=0$.
		Let $c_{00}$ denote the constant coefficient of $P(x,y)$.
		Since $p^u$ divides $P(px,py)$, it follows that $p^u$ divides $c_{00}$.
		By iteration, $p^{ku}$ divides the $k$-th iterate $P(p^kx,p^ky)$, and hence also divides $c_{00}$.
		This holds for all $k$, therefore $c_{00}=0$.
		By hypothesis $P(x,y) = F(x) + G(y)$, and we may thus assume without loss of generality that
		$F(0)=0$ and $G(0)=0$.
		
		If $(0,0)$ has thickness $t \ge u$ for $P(x,y)$, then
		$p^u \mid F(px) + G(py)$.
		Since $F(0)=0$ and $G(0)=0$, it follows that $p^u \mid F(px)$ and $p^u \mid G(py)$.
		Write $F(x) = p^\beta \sum_i a_i x^i$ and $G(y) = p^\alpha \sum_i b_i y^i$, with $a_0=0$ and $b_0=0$.
		Since $p$ does not divide $P(x,y)$, we may assume $\beta=0$.
		
		After $k$ iterations, $p^{ku}$ divides $F(p^kx)$, hence
		$p^{ku}$ divides $F(p^kx) = \sum_i a_i p^{ki} x^i$,
		and similarly $p^{ku}$ divides
		$G(p^ky) = p^\alpha \sum_i b_i p^{ki} y^i$.
		Thus $p^{ku}$ divides $a_i p^{ki}$ and $b_i p^{\alpha+ki}$ for all $i \ge 0$.
		Letting $k$ tend to infinity, we obtain
		$a_i=0$ and $b_i=0$ for $i=0,\ldots,u-1$.
		Thus $F(x) = x^u U(x)$ and $G(y) = p^\alpha y^v V(y)$ with $v \ge u$.

		\item We have $F(x) = x^u U(x) = x^u (u_0+u_1x+u_2x^2+\cdots)$.
		We compute $F(px)$ to obtain the successor:
		$F_1(x) = x^u(u_0 + u_1 p x + u_2 p^2 x^2 +\cdots)$.
		After $m$ iterations, $p^m$ divides the coefficients of $x$, $x^2$, \ldots.
		The proof is similar for $V(y)$.
		
		\item We have seen that $P(x,y)=x^uU(x)+p^\alpha y^vV(y)$,
		with $v \ge u$.
		First observe that if $U(0) \equiv 0\pmod{p}$,
		then by the univariate trunk algorithm we have
		$U(x)=x^{u'}\tilde{U}(x)$ with $u'>u$.
		In the cases $v>u$ or $\alpha>0$, this is excluded, since it would give
		a thickness of $P(x,y)$ at $(0,0)$ strictly larger than $u$.
		In these cases we therefore have $U(0)\not\equiv0\pmod{p}$.
		In the remaining case we have
		$P(x,y)=x^{u}U(x)+y^{u}V(y)$ with $U(0) \equiv 0\pmod{p}$.
		We cannot also have $V(0)\equiv0\pmod{p}$, since this would imply
		$V(y)=y^{v'}\tilde{V}(y)$ and again a thickness of $P(x,y)$ at $(0,0)$
		strictly larger than $u$.
		Thus we have $U(0)\not\equiv0\pmod{p}$ or $V(0)\not\equiv0\pmod{p}$.
		By possibly exchanging the variables, we may assume
		$U(0) \not\equiv 0\pmod{p}$.
		
		We now start again from
		$P(x,y)=x^u U(x)+p^\alpha y^vV(y)$ with $v\ge u$.
		Applying the one-variable thickness algorithm to $V(y)$ along the principal stalk,
		the thickness becomes constant (possibly zero) from some height $e$ onward.
		Thus, by applying in one variable the reasoning we have just carried out,
		we obtain $V_e(y) = y^w W(y)$ with
		$W(0) \not\equiv 0 \pmod{p}$.
		At height $e$, we have
		$P_{e}(x,y)=x^uU_e(x)+p^{\alpha'} y^{v}V_e(y)$, hence
		$P_{e}(x,y)=x^uU_e(x)+p^{\alpha'}y^{v+w}W(y)$,
		with $W(0)\not\equiv0\pmod{p}$,
		which is consistent with the final conclusion of the lemma.
		
	\end{enumerate}	
\end{proof}

\subsection{Finiteness of infinite stalks}

\begin{lemma}
	\label{lem:finite}
	The number of infinite stalks issued from the root of the trunk, whose vertices all have thickness $t\ge 2$, is finite.
\end{lemma}

As an immediate consequence, the number of stalks of constant thickness $\ge 2$ (issued from any vertex of the trunk) is finite.

\begin{proof}
	First note that the thickness of a vertex is bounded above by the degree of the polynomial, and hence can take only finitely many values.
	Let $P(x,y)$ denote the polynomial, and let $(\rho_i, \sigma_i) \in\Zz_p^2$ be the vertices at infinity of the infinite stalks.
	By Lemmas~\ref{lem:translation} and~\ref{lem:normnew}, after normalization we have (with $t=u_i$)
	\[
	P(x,y) = (x-\rho_i)^{u_i} U_i(x) + p^{\alpha_i} (y-\sigma_i)^{v_i} V_i(y)
	\]
	for each $i$.
	By identification, we have
	$(x-\rho_i)^{u_i} U_i(x) = (x-\rho_1)^{u_1} U_1(x) + c_{i}$ (where $c_{i}$ is a constant).
	Since $u_i \ge 2$, differentiation shows that $\rho_i$ is a root of the fixed polynomial
	$\frac{d}{dx}(x-\rho_1)^{u_1} U_1(x)$, and hence can take only finitely many values.
	The same argument applies to the $\sigma_i$.
	Thus the set of $(\rho_i, \sigma_i)$ is finite, and therefore only finitely many stalks issue from the root of the trunk.
\end{proof}

\section{The generating function is rational}
\label{sec:rational}

In this section we prove Theorem~\ref{th:igusa-intro}: the Poincaré series of a separated polynomial in two variables is rational.
The idea is to establish this result for certain model polynomials, as we began to do at the end of Section~\ref{sec:count}. This time, starting from an arbitrary polynomial, we decompose its trunk into a finite number of pieces, each of which is either finite or isomorphic to the trunk of a model polynomial. From explicit computations for each piece, we will deduce that the Poincaré series for the polynomial under consideration is a rational function, which is precisely Igusa’s result.


If $P$ is degenerate, that is, if for instance $P(x,y)=R(x)$ with $R \in \Zz_p[X]$, then the trunk of $P(x,y)$ is easily deduced from that of $R(x)$, and the Poincaré series of $P(x,y)$ coincides with that of $R(x)$. It follows from the study of the univariate case in \cite{BD26}, using the same arguments as in the present paper, that the Poincaré series of $R(x)$ (and hence that of $P(x,y)$) is rational. In this case, one recovers Igusa's result. Consequently, in the sequel, we assume that $P(x,y)$ actually depends on both $x$ and $y$.

\subsection{Reduction of the proof of rationality to the normalized case sufficiently far along an infinite stalk.}
\label{ssec:reduc}

We work under the assumptions of Theorem~\ref{th:igusa-intro}.
The idea is to establish the result first for certain model polynomials, as we began to do at the end of Section \ref{sec:normal}. We now proceed as follows: starting from an arbitrary polynomial, we decompose its trunk into finitely many pieces, each of which is either finite or isomorphic to the trunk of a model polynomial. From explicit computations on each piece, we then deduce that the Poincaré series associated with the polynomial under consideration is a rational function, thereby recovering Igusa’s result.

If the trunk of the polynomial $P(x,y)$ is finite, then by Theorem~\ref{th:xylon}, the Poincaré series $S(T)$ is a polynomial, which proves the desired result in this case.
If the trunk is infinite, then by Lemma~\ref{lem:finite}, there exist only finitely many (maximal) infinite stalks of constant thickness $\ge 2$.

By cutting these constant-thickness stalks high enough up and keeping only the vertices above a certain level, we may assume they are independent, meaning that no ancestor of a vertex in one stalk belongs to another stalk.
Let $\Sigma_i$ be such a stalk. 
For each such stalk $\Sigma_i'$, we denote by $\mathrm{T}_i$ the part 
of the trunk of $P$ attached to it, that is, the set of vertices of 
$\Trunk(P)$ having an ancestor in $\Sigma_i'$. 


The main task of this section is to prove that the contribution of each 
$\mathrm{T}_i$ to the series $S(T)$ is rational. We denote by 
$\mathrm{T}'$ the trunk of $P$ with all the $\mathrm{T}_i$ removed.

There may exist infinitely many stalks (infinite paths) of thickness $1$.
However, we saw in Section~\ref{ssec:hensel} that for a vertex of thickness $1$, either it is a terminal vertex (the trunk stops there), or a Hensel tree is attached to it (which indeed produces infinitely many stalks of thickness $1$).
And we know that a Hensel tree contributes a rational fraction (Lemma~\ref{lem:Shensel}, see also the end of Section~\ref{ssec:Srat}).

Let $\mathrm{T}''$ denote the tree obtained by removing from $\mathrm{T}'$ all Hensel trees.
Such a tree is finite (otherwise, by König’s lemma, it would contain an infinite stalk of thickness $1$ or $\ge 2$, both of which have been removed by construction).
Since $\mathrm{T}''$ is finite, only finitely many Hensel trees can be attached to it, and their total contribution is therefore a rational series.
(Warning: as we shall see in the case of infinite stalks of thickness $\ge 2$, infinitely many Hensel trees may be attached to such a stalk, but these stalks are by definition excluded from the tree $\mathrm{T}'$.)

\medskip

We therefore consider an infinite stalk of constant thickness $\ge 2$ (among finitely many possible ones, see Lemma~\ref{lem:finite}) issued from a vertex $s_0$ of height $k_0$ and with tree-top value $\varphi_0$.
We denote by $P_0(x,y)$ the polynomial attached to the root of this stalk; it is again a separated polynomial.
We consider two series:
\begin{itemize}
	\item the Poincaré series $S_{|}(T)$ associated with the original polynomial $P(x,y)$ and with the part of the trunk issued from the vertex $s_0$;
	\item the Poincaré series $S_0(T)$ associated with the polynomial $P_0(x,y)$ (that is, we now consider $s_0$ as the root of the trunk).
\end{itemize}
Then the xylon formula~\ref{th:xylon} gives the relation
$S_{|}(T) = \frac{T^{\varphi_0}}{p^{2k_0}} S_0(T)$
(see also the remark following Lemma~\ref{lem:Shensel}).
It therefore suffices to prove the rationality of $S_0(T)$.

Still considering the case of an infinite stalk of constant thickness $\ge 2$, we use the results of Section~\ref{sec:normal} to reduce to a simple form of $P_0(x,y)$.
More precisely, we work under the conclusions of Lemma~\ref{lem:normnew}:
\[
P_0(x,y) = x^u U(x) + p^\alpha y^v V(y).
\]

We may even, without loss of generality and still by Lemma~\ref{lem:normnew}, consider the polynomial sufficiently high on the stalk, and thus assume that
\[
U (x) \equiv u_0 \pmod {p^m}, 
\qquad
V(y) \equiv v_0 \pmod{p^m},
\]
with $m$ an arbitrarily large integer and
$u_0 \not\equiv 0 \pmod{p}$, $v_0 \not\equiv 0 \pmod{p}$.
We will even see that we can further reduce to
$Q_0(x,y) = u_0x^u+p^\alpha v_0y^v$.
It is for these polynomials that we will carry out the computations of the Poincaré series.

\subsection{Normalized case}

\begin{proposition}
	\label{prop:rat}
	The series $S(T)$ associated with $Q_0(x,y) = u_0 x^u + p^\alpha v_0 y^v$ is rational
	(where $u_0 \not\equiv 0 \pmod{p}$, $v_0 \not\equiv 0 \pmod{p}$ and $1 < u \le v$).
	
	The trunk associated with
	$P_0(x,y) = x^u U(x) + p^\alpha y^v V(y)$
	(with $\alpha \ge 0$, $1 < u \le v$, $U(x) = u_0+u_1x+\cdots$, $V(y) = v_0+v_1y+\cdots \in \Zz_p[y]$,
	and
	$u_0 \not\equiv 0 \pmod{p}$, $v_0 \not\equiv 0 \pmod{p}$, and for all $i\ge1$,
	$u_i \equiv 0 \pmod{p^m}$, $v_i \equiv 0 \pmod{p^m}$ with $m=\val_p(u)+2$)
	is identical to that of $Q_0(x,y)$.
	In particular, the same sum $S(T)$ is associated with it and is therefore rational.
	
	Moreover, the denominators of this series in $\Qq(T)$ are among:
	\[ 
	p -T, \quad p^2- T^u, \quad p^{2u} - T^{u^2}, \quad p^{u+v}-T^{uv}.
	\] 
\end{proposition}

The remainder of this section is devoted to the proof.
Our proof is explicit: we partition the trunk into a finite number of subtrees, and for each such subtree we give the exact expression of the associated rational series.
Since the formulas are tedious, we do not give the explicit expression of the final series.
However, a complete example will be worked out in the Appendix. 

\subsection{Description of the trunk after leaving the principal stalk}

We begin with the general case
$P_0(x,y) = x^u U(x) + p^\alpha y^v V(y)$.
The root $(0,0)$ gives rise to an infinite stalk of constant thickness $t=u$, called the \emph{principal stalk}.
The $k$-th successor of $P_0$ is
\[
P_k(x,y) = x^u U(p^kx) + p^{\alpha+k(v-u)} y^v V(p^ky),
\]
which has a form similar to that of $P_0$.
The contribution of this stalk to the sum $S(T)$ is rational by Lemma~\ref{lem:Sstalk}.

\subsection{The case $\alpha > 0$}
\label{ssec:alphapositive}

We first reduce the case $\alpha > 0$ by Euclidean division to one of the two cases
$0 < \alpha < u$ or $\alpha = 0$. We then treat these two situations.

The reduction modulo $p$ of $P_0$ is
$\bar P_0(x,y) = u_0x^u$, whose roots are of the form $(0,s_0)$.
We have already studied the case $s_0=0$ above, so we assume $s_0 \not\equiv 0 \pmod{p}$.
At such a root $(0,s_0)$, we have
\[
P_0(0+px,s_0+py) = p^u x^u U(px) + p^\alpha (s_0+py)^v V(s_0+py).
\]
The constant term of $(s_0+py)^v V(s_0+py)$ is
$\tilde v_0 = v_0 s_0^v$. 

\medskip

\emph{Case $\alpha \ge u$.}
When $\alpha \ge u$, each root $(0,s_0)$ has thickness $u$, and the successor of $P_0$ is:
\[
P_1(x,y) = x^u U(px) + p^{\alpha-u} (s_0+py)^v V(s_0+py).
\]
As long as the exponent of $p$ is $\ge u$, we continue this process, the $k$-th successor being:
\[
P_k(x,y) = x^u U_k(x) + p^{\alpha-ku} (s_0+s_1p+\cdots+ s_{k-1}p^{k-1}+p^ky)^v V_k(y),
\]
where
$U_k(x) = U(p^kx)$ and
$V_k(y) = V(s_0+s_1p+\cdots +s_{k-1}p^{k-1}+p^ky)$.
We still have
$U_k(x) \equiv u_0 \pmod {p^m}$ and
$V_k(y) \equiv v_0 \pmod{p^m}$.

Thus we obtain $\lfloor \frac{\alpha}{u} \rfloor$ iterations,
leading to a tree that splits at the first step into
$p - 1$ vertices, and then splits in cascade, each vertex thereafter having $p$ children,
over a total height of $\lfloor \frac{\alpha}{u} \rfloor$,
with vertices all having thickness equal to $u$.
(For the initial vertex there is one omitted edge, since it is already counted in the stalk above $(0,0)$.)

By Euclidean division, we are thus reduced to the case
$0 < \alpha' < u$ or to the case $\alpha' = 0$,
where $\alpha'$ is the remainder of the Euclidean division of $\alpha$ by $u$.

\medskip

\emph{Case $0 < \alpha' < u$.}

Let us restart from
\[
P_k(x,y) = x^u U_k(x) + p^{\alpha'} (s_0+s_1p+\cdots +s_{k-1}p^{k-1}+p^ky)^v V_k(y)
\]
with $0 < \alpha' < u$.
A root $(0,s_k)$ has thickness $t=\alpha'$, and the successor is:
\[
P_{k+1}(x,y) = p^{u-\alpha'} x^u U_k(px) + (s_0+s_1p+\cdots +s_kp ^k + p^{k+1}y)^v V_k(s_k+py).
\]
Thus
$\bar P_{k+1}(x,y) \equiv s_0^v v_0 \equiv \tilde v_0 \not\equiv 0 \pmod{p}$.
Hence $\bar P_{k+1}$ has no root modulo $p$, and this branch terminates.

\medskip

Writing $U(x) = u _0 + u_1x+\cdots$ and $V(y) = v_0 + v_1y+\cdots$,
one sees, upon repeating the previous computations, that they do not depend on the coefficients
$u_1, u_2,\ldots$ or $v_1, v_2,\ldots$.
Thus the trunk for $P_0$ is identical to that of
$Q_0 = u_0 x^u + p^\alpha v_0 y^v$.

\medskip

Summary: in the case where $u$ does not divide $\alpha$, the trunk consists of a tree that first has
$p-1$ branches, which then split into $p$ branches, then again into $p$ branches, and so on,
for a total of $\lfloor \frac{\alpha}{u} \rfloor$ steps, always with thickness $u$.
Then there is a final splitting into $p$ branches, with terminal vertices of thickness
$\alpha' = \alpha \% u$ (the remainder of the Euclidean division of $\alpha$ by $u$).

\medskip

It remains to consider the case where $u$ divides $\alpha$.
As we have seen, this situation reduces to the case $\alpha=0$, which we now study.
Beforehand, we prove a few results on valuations in the binomial expansion of Newton.

\subsection{Valuations of the binomial coefficients}

\begin{lemma}
	\label{lem:binomial}
	Let $p$ be a prime number.
	\begin{enumerate}
		\item If $k \ge 1$, then $\val_p( k! ) \le k-1$.
		\item If $1 \le k \le n$, then $\val_{p}\binom{n}{k} \ge \val_{p}(n)+1-k$. 
		\item If $2 \le k \le n$, then $\val_{p}\!\left[k\binom{n}{k}\right] \ge \val_{p}(n)+2-k$.
		\item Let $l\ge 1$, $n \ge 1$, and let $s$ be an integer not divisible by $p$.
		Denote by $C_k$ the coefficient of $x^k$ in the expansion of $(s+p^lx)^n$. 
		Then for $k\ge1$, one has $\val_p(C_k) \ge \val_p(C_1)$. 
	\end{enumerate}
\end{lemma}

\begin{proof}
	\begin{enumerate}
		\item By Legendre’s formula,
		\[
		\val_p( k! ) = \left\lfloor \frac{k}{p} \right\rfloor + \left\lfloor \frac{k}{p^2} \right\rfloor + \cdots
		\]
		In particular,
		\[
		\val_p( k! ) 
		\le \val_2 (k!) 
		= \sum_{j \ge 1} \left\lfloor \frac{k}{2^j} \right\rfloor
		<  \sum_{j \ge 1}\frac{k}{2^j} 
		= k,
		\]
		the second inequality being strict since the terms
		$\left\lfloor \frac{k}{2^{j}}\right\rfloor$
		vanish from some index onward.
		Thus $\val_p( k! ) \le k-1$.
		
		\item The integer $\binom{n}{k}$ is the quotient of
		$n(n-1)\cdots(n-k+1)$ by $k!$.
		It follows that
		$\val_{p}\binom{n}{k} \ge \val_{p}(n)-\val_{p}(k!)$.
		Thus, by the first property,
		$\val_{p}\binom{n}{k} \ge \val_{p}(n)+1-k$. 
		
		\item Since $k\binom{n}{k} = n\binom{n-1}{k-1}$, we have:
		\begin{align*}
			\val_{p}\!\left[k\binom{n}{k}\right] 
			&= \val_{p}\!\left[n\binom{n-1}{k-1}\right] 
			= \val_{p}(n) + \val_{p} \binom{n-1}{k-1} \\
			&\ge \val_p(n) + \big( \val_{p}(n-1)+1-(k-1) \big) \\
			&\ge \val_p(n) + 2 - k,
		\end{align*}
		where we used the inequality from the previous point, applied to
		$\binom{n-1}{k-1}$.
		
		\item We have
		$\val_p(C_1) = \val_p (ns^{n-1}p^l) = \val_p(n) + l$,
		whereas
		$\val_p(C_k) = \val_p (\binom{n}{k}s^{n-k}p^{lk})$.
		By the second point,
		\[
		\val_p(C_k) \ge \big( \val_p(n) + 1 - k \big) + lk
		= \val_p(C_1) + (k-1)(l-1)
		\ge \val_p(C_1).
		\]
	\end{enumerate}
\end{proof}

\subsection{Case $\alpha = 0$}
\label{ssec:alphazero}

We resume the notation of Section~\ref{ssec:alphapositive}, now with $\alpha=0$:
\[
P_k(x,y) = x^u U_k(x) + (s_0+s_1p+\cdots+ s_{k-1}p^{k-1}+p^ky)^v V_k(y),
\]
where $U_k (x) \equiv u_0 \pmod {p^m}$ and $V_k(y) \equiv v_0 \pmod{p^m}$.
This formula is valid for $k=0$, in which case we simply have
$P_0(x,y) = x^u U(x) + y^vV(y)$.
If $P_k$ has no root modulo $p$, then the trunk ends at that vertex.
Otherwise, let $(r_k,s_k)$ be a root of $P_k$ modulo $p$.
Modulo $p^m$ we have:
\[
P_k(r_k+px, s_k+py) 
\equiv u_0(r_k+px)^u + v_0(s_0+s_1p+\cdots +s_kp^k+p^{k+1}y)^v \pmod{p^m}.
\]
We set $r=r_k$, $s = s_0+s_1p+\cdots +s_kp^k$, and $l=k+1$ in order to write more simply:
\[
P_k(r_k+px, s_k+py) 
\equiv u_0(r+px)^u + v_0(s+p^ly)^v \pmod{p^m}.
\]

Thus:
\begin{align*}
	P_k(r_k+px, s_k+py) 
	&\equiv u_0(r+px)^u + v_0(s+p^ly)^v \pmod{p^m}\\
	&\equiv u_0r^u+v_0s^v \\
	& \qquad + u_0u r^{u-1} px + u_0\binom{u}{2}r^{u-2}p^2x^2 + \cdots \\
	& \qquad \qquad + v_0v s^{v-1} p^ly  + v_0\binom{v}{2}r^{v-2}p^{2l}y^2 + \cdots \quad \pmod{p^m}.
\end{align*}

Let $\beta = \val_p(u_0r^u+v_0s^v)$ and $\gamma = \min(\val_p(pu), \val_p(vp^l))$.
Note that by Lemma~\ref{lem:binomial}, the valuations of the coefficients of $x^2$, $x^3$, \ldots, $y^2$, $y^3$, \ldots{} are greater than or equal to $\gamma$.
We will see that the thickness of the root $(r_k,s_k)$ equals $t=\min(\beta, \gamma)$.
Thus, by definition of $\gamma$, this thickness is less than or equal to $u+1$, hence strictly less than $m$, which justifies that all computations may be carried out modulo $p^m$.

\medskip

\emph{First case: $\beta < \gamma$.}
Note that $\beta>0$, since by definition $(r_k,s_k)$ is a root of $P_k$ modulo $p$.
Assume now that $0 < \beta < \gamma$.
We prove that in this case the vertex $(r_k, s_k)$ has thickness $\beta$ and has no successor in the trunk (i.e.~the trunk terminates there).
Since $0 < \beta < \gamma$, the thickness of the root $(r_k,s_k)$ equals $\beta$.
Thus:
\[
P_{k+1}(x,y) = \frac{1}{p^\beta}P_k(r_k+px, s_k+py) 
= \frac{u_0r^u+v_0s^v }{p^\beta} + pR(x,y).
\]
Indeed, all coefficients of
$u_0u r^{u-1} px+\cdots + v_0v s^{v-1} p^ly  +\cdots$
are divisible by $p^\gamma$, hence by $p^{\beta+1}$.
Thus the reduction modulo $p$ of $P_{k+1}$ is reduced to a constant polynomial:
\[
\reduc{P}_{k+1}(x,y) \equiv  \frac{u_0r^u+v_0s^v }{p^\beta} \pmod{p}.
\]
By definition of $\beta$, this constant is nonzero modulo $p$.
Hence $P_{k+1}$ has no roots modulo $p$, and therefore the trunk stops here.

\medskip

\emph{Second case: $\beta \ge \gamma$.}
We prove that in this case, a Hensel tree arises from each of the possible roots $(r_{k+1},s_{k+1})$ succeeding $(r_k,s_k)$.

Since $\beta \ge \gamma$, the thickness of the root $(r_k,s_k)$ now equals $\gamma$.
Thus the successor of $P_k$ at $(r_k,s_k)$ can be written as:
\begin{align*}
	P_{k+1}(x, y) 
	&\equiv \frac{1}{p^\gamma}(u_0r^u+v_0s^v) \\
	& \qquad + \frac{1}{p^\gamma}(u_0u r^{u-1} p)x + \frac{1}{p^\gamma}(u_0\binom{u}{2}r^{u-2}p^2)x^2 + \cdots \\
	& \qquad \qquad + \frac{1}{p^\gamma}(v_0v s^{v-1} p^l)y  + \frac{1}{p^\gamma}(v_0\binom{v}{2}s^{v-2}p^{2l})y^2 + \cdots \quad \pmod{p^{m-\gamma}}.
\end{align*}

We write this as:
\begin{align*}
	P_{k+1}(x,y) 
	&= \frac{1}{p^\gamma}(u_0r^u+v_0s^v) \\
	& \quad + c_1 x + c_2x^2+ \cdots \\
	& \quad \quad + d_1 y + d_2y^2 +\cdots .
\end{align*}

Recall that $\gamma = \min(\val_p(pu), \val_p(vp^l))$.
Assume $\gamma = \val_p(vp^l)$; then $d_1 \not\equiv 0 \pmod{p}$.
Let us compute $\frac{\partial P_{k+1}}{\partial y}(x,y)$.
The constant coefficient of this partial derivative is $d_1$.
The coefficient of $y^{j-1}$ in this partial derivative is:
\[
j d_j = j v_0\binom{v}{j}r^{v-j}p^{jl}.
\]
By Lemma~\ref{lem:binomial}:
\begin{align*}
	\val_p(j d_j) 
	&= \val_p \left( j\binom{v}{j} \right) + jl \\
	&\ge (\val_p(v)+2-j) + jl \\
	&\ge (\val_p(v)+l) + 1 + (j-1)(l-1) \\
	&> \gamma.
\end{align*}
Thus
$\frac{\partial P_{k+1}}{\partial y}(x,y) = d_1 +pR(y)$.
Let $(r_{k+1},s_{k+1})$ be a solution of $P_{k+1}(x,y) \equiv 0 \pmod{p}$.
Still modulo $p$, we have
$\frac{\partial P_{k+1}}{\partial y}(x,y) \equiv d_1 \not\equiv 0 \pmod{p}$,
in particular at $(x,y) = (r_{k+1},s_{k+1})$.
Therefore the thickness of $(r_{k+1},s_{k+1})$ is exactly $t=1$, and we are in the situation of Hensel’s lemma.

If $\gamma = \val_p(pu)$, then $c_1 \not\equiv 0 \pmod{p}$.
A very similar computation (taking $l=1$) gives
$\frac{\partial P_{k+1}}{\partial x}(x,y) \equiv c_1 \not\equiv 0 \pmod{p}$,
and we are again in the situation of Hensel’s lemma.

\subsection{General structure of the outgoing branches}
\label{ssec:higher}

To prove the rationality of the sum $S(T)$, the trunk must have a particular structure.
We will see that when one moves up along the principal stalk, that is, above the roots $(0,0)$, the branches attached to it form patterns that are roughly periodic, above a certain height.

We set $A_0(x,y) = P_0(x,y) = x^u U(x) + p^\alpha y^vV(y)$.
Let $k>0$ be a multiple of $u$; we denote by $B_0(x,y) = P_k(x,y)$ the $k$-th successor of $P_0$ along the stalk $(0,0)$:
\[
B_0(x,y) = P_k(x,y) = x^u U(p^kx) + p^{\alpha + k(v-u)} y^v V(p^ky).
\]

In this subsection, we assume that $v-u>0$, so that the quantity 
$\alpha + k(v-u)$ becomes arbitrarily large as $k \to \infty$. 
We will see later that the case $v=u$ is much simpler.

The polynomial $B_0$ is similar to $A_0$, where the exponent $\alpha_k = \alpha + k(v-u)$ changes,
but the exponents $u$ and $v$, and the constant coefficients $u_0$ and $v_0$, do not.
Now, the computations in the previous paragraphs depend mainly on these data $\alpha_k, u, v, u_0, v_0$.
We must now show that, above a certain height, the outgoing branch from the principal stalk for $A_0$ (at height $0$)
is similar to the outgoing branch for $B_0$ (at height $k$).
According to the computations in the previous paragraphs, the branch leaving a height $k$ begins with branches of thickness $t=u$, of length
$\lfloor \frac{\alpha_k}{u} \rfloor$, which therefore becomes longer and longer as $k$ grows.
Note that since we choose here $k$ to be a multiple of $u$, we have
\[
\left\lfloor \frac{\alpha_k}{u} \right\rfloor
= \left\lfloor \frac{\alpha}{u} \right\rfloor + \frac{k}{u}(v-u).
\]

\medskip

We now show that apart from this difference in branch length, what is attached afterwards is identical at height $0$ and at height $k$.

If $\alpha$ is not a multiple of $u$, then at the end of the stalk for $A_0$ a vertex of thickness $\alpha \% u$ is attached (see Section~\ref{ssec:alphapositive}).
In this case $\alpha_k$ is also not a multiple of $u$, and since $\alpha_k \% u = \alpha \% u$, at the end of the stalk for $B_0$ a vertex of thickness $\alpha \% u$ is also attached.

Now consider the case where $\alpha$, and hence also $\alpha_k$, are multiples of $u$.
This corresponds to the situation of Section~\ref{ssec:alphazero}, whose notation we reuse.
We denote by $A_i$ and $B_i$ the respective successors of $A_0$ and $B_0$ along their outgoing branch, corresponding to the same choices of the roots $s_0,s_1,\ldots,s_{i-1}$ for both polynomials.
Let $A_m$ and $B_n$ be the successors corresponding to the exponent $\alpha'=0$
(with $m = \lfloor \frac{\alpha}{u} \rfloor < n = \lfloor \frac{\alpha_k}{u} \rfloor$),
that is, to the last vertex of thickness $u$ on the stalk.

For the subsequent reasoning, it will suffice to study the highest part 
of the stalk. Consequently, we assume that $\alpha$ is sufficiently large 
so that the inequality $\val_{p}(v p^{m+1}) \ge \val_{p}(p u)$ holds, 
where $m = \left\lfloor \frac{\alpha}{u} \right\rfloor$. 
We therefore obtain $\gamma(A_{m}) = \val_{p}(p u) = \val_{p}(u) + 1$. 
Moreover, since $n > m$, we also have 
$\val_{p}(v p^{n+1}) \ge \val_{p}(p u)$, and hence 
$\gamma(B_{m}) = \val_{p}(u) + 1$.
Therefore $\gamma(B_n) = \gamma(A_m)$, which we simply denote by $\gamma$.
Next, note that the roots of $A_m$ modulo $p$ and those of $B_n$ modulo $p$ are the solutions of the same equation $u_0 x^u + v_0 s_0^v \equiv 0 \pmod{p}$.
Choose $(\rho, \sigma)$ to be such a root.
In the ``$\alpha'=0$'' situation, the point is to compare the value of $\beta$ with $\gamma$.
Let
\[
C = u_0 \rho^u + v_0(s_0+s_1p+\cdots+ s_{m-1}p^{m-1}+ \sigma p^m)^v,
\]
so that $\beta(A_m) = \val_p(C)$.
Also let
\[
C' = u_0 \rho^u + v_0(s_0+s_1p+\cdots+ s_{m-1}p^{m-1}+ \cdots+ s_{n-1}p^{n-1} +\sigma p^n )^v,
\]
so that $\beta(B_n) = \val_p(C')$.
If we move sufficiently high on the principal stalk so that
$m = \lfloor \frac{\alpha}{u} \rfloor \ge \val_p(u)+1 = \gamma$,
then we have the equivalences:
\[
\beta(A_m) \ge \gamma 
\iff C \equiv 0 \pmod{p^{\val_p(u)+1}}
\iff C' \equiv 0 \pmod{p^{\val_p(u)+1}}
\iff \beta(B_n) \ge \gamma.
\]
Thus, if $\beta(A_m) \ge \gamma$, then we also have $\beta(B_n) \ge \gamma$, and we attach a Hensel tree in both cases.
If $\beta(A_m) < \gamma$, then we also have $\beta(B_n) < \gamma$, and the trunk stops at the end of the stalk of thickness $\beta$ in both cases.

Finally, in the case $\beta < \gamma$, after $\gamma$ vertices after leaving the principal stalk, the value of $\beta$ is constant.
Indeed, let
\[
C_k = u_0 \rho^u + v_0(s_0+s_1p+\cdots+ s_kp^{k})^v.
\]
Since we assume $\beta < \gamma$, for $k \ge \gamma$ we have
\[
\beta_k = \val_p(C_k) = \val_p(C_\gamma).
\]

\subsection{Decomposition of the Poincaré series}
\label{ssec:Srat}

We have already seen two simple cases where the series is rational:
(i) the case of an infinite stalk of constant thickness $u$ (Lemma~\ref{lem:Sstalk}), which produces a denominator $p ^2-T^u$,
and (ii) the case of a Hensel tree issued from the origin (Lemma~\ref{lem:Shensel}), which produces a denominator $p-T$.

\subsection*{A. Series for the principal stalk and a particular case}

The vertices of this principal stalk, of constant thickness $u$, yield the rational series (Lemma~\ref{lem:Sstalk}):
\[
S_A(T) = 1 \  + \ A(T) \frac{T}{p^2} \left( 1-\frac{T^u}{p^2} \right)^{-1}
\]
where we set:
\[
A(T) = 1+T+T^2+\cdots + T^{u-1}.
\]

\bigskip
\bigskip

We now consider the trunks encountered in this section. We will prove rationality by using periodicity.
We have seen that the exponent $\alpha$ defining $P_0$ (or $Q_0$) becomes, for the $k$-th successor $P_k$ along the principal stalk, $\alpha_k = \alpha + k(v-u)$.
Modulo $u$, this exponent is therefore periodic with period (at most) $u$.

For any set $E$ of vertices, we define $\Poinc(E)$ as the series of the values $\Xi_s(T)$ of the xylon function for all the elements $s$ of $E$.
In this language $S_A(T) = \Poinc( \mathrm{PrincStalk} )$.

We define the \defi{starred trunk} $\Trunk^*_k$ as the set of vertices whose most recent ancestor on the principal trunk at height $k$, i.e.{} the set of branches leaving the principal trunk is at height $k$. Formally this is the set of the vertices $\left(r,s,j\right)$ of the trunk 
such that $(0,0,k) \vartriangleleft (r_{k+1},s_{k+1},k+1) \vartriangleleft (r,s,j)$
where $\left(r_{k+1},s_{k+1}\right)\neq (0,0)$.

The Poincaré series of the whole trunk is:
\[
S(T) = S_A(T) + \sum_{k=0}^{+\infty} \Poinc(\Trunk^*_k).
\]

In this expression, none of the starred trunks $\Trunk^*$ can contain infinitely 
many (maximal) Hensel trees, nor an infinite stalk of thickness $\ge 2$ 
(because we have build independent infinite constant stalks in paragraph 
\ref{ssec:reduc}). It follows that each series $\Poinc(\Trunk^*)$ is either 
a polynomial or a rational function whose only possible non-constant 
denominator is $p - T$.

The case $v=u$, that is $P_0(x,y) = x^u U(x) + p^\alpha y^u V(y)$, 
is a special situation for which the general method used when $v>u$ 
is not relevant. We therefore state a specific proposition for this case.

\begin{proposition}
	In the case $P_0(x,y) = x^u U(x) + p^\alpha y^u V(y)$, 
	the Poincaré series of the polynomial is a rational function.
\end{proposition}

\begin{proof}
	In this case, each successive transform $P_k(x,y)$ of $P_0(x,y)$ 
	has the form $P_k(x,y) = x^u U(p^kx) + p^\alpha y^u V(p^ky)$, 
	which is essentially stable. It follows that all the starred trunks 
	$\Trunk_k^*$ are isomorphic to one another. Ignoring their initial 
	level, they therefore have the same Poincaré series, say $H(T)$, 
	which is a rational function by the previous remark.
	
	From the earlier relation we obtain 
	$\Poinc(\Trunk_k^*) = \frac{T^{ku}}{p^{2k}} H(T)$. 
	Using this identity, we deduce 
	\[
	S(T) = S_A(T)+\sum_{k=0}^{+\infty}\frac{T^{ku}}{p^{2k}}H(T)
	= S_A(T)+\left(1-\frac{T^u}{p^2}\right)^{-1}H(T),
	\]
	which proves the proposition and, in particular, Proposition~\ref{prop:rat} 
	in this case.
\end{proof}

From now on, we assume that $v > u$. In this case, in the formula above
$S(T) = S_A(T) + \sum_{k=0}^{+\infty} \Poinc(\Trunk^*_k)$,
we decompose the second summand as
\[
\sum_{k=0}^{u-1} \sum_{l \ge 0} \Poinc(\Trunk^*_{k+lu}).
\]
To prove the rationality of $S(T)$, it is sufficient to prove the rationality 
of each series 
\[
\sum_{l \ge 0} \Poinc(\Trunk^*_{k_0+lu})
\]
for $k_0 = 0, \ldots, u-1$.

But we have already established that each series 
$\Poinc(\Trunk^*_{k_0+lu})$ is either a polynomial or a rational function 
whose only possible denominator is $p - T$.

Consequently, in order to prove rationality, it suffices to show that 
there exists an integer $L$ such that 
$\sum_{l \ge L} \Poinc(\Trunk^*_{k_0+lu})$ is rational. 
In other words, to prove rationality, we may assume that we are at a sufficiently 
large level along the principal stalk.

We choose $L$ any integer such that $\alpha+Lu(v-u) \ge \gamma u$, and set $a=\alpha+Lu(v-u)$. 
We will prove the rationality of the Poincaré series for the polynomial $\tilde P_0$ with this exponent $a$ (instead of the original $P_0$ of exponent $\alpha$)
using the fact that $a \ge \gamma u$ (see parts D and E below).
To this end, we decompose the vertices into two sets: on one hand,
the set of the interior vertices of all the starred trunks, with thickness
$u$; on the other hand the terminal vertices of all these starred
trunks. In the case $a\%u\neq0$, these vertices have thickness $a\%u$; 
in the other case $a\%u=0$, these terminal vertices may be either
simple, or belonging to a Hensel tree.

\subsection*{B. Series for all interior vertices}

Consider an interior vertex of an outgoing branch (among the $p-1$ possible ones) of a starred trunk $\Trunk^*(P_{lu})$. It is characterized by the integer $l$ and another integer $m$ which is the distance between the vertex and the departure vertex of this starred trunk.
The height of this general vertex is $k = lu+m$, its thickness is $t=u$, and the value of the tree-top function is $\varphi = (lu+m)u$.

Thus, the contribution of this vertex is:
\[
\Xi_{l,m}(T) 
= \frac{T^{\varphi-t+1}}{p^{2k}}(1+T+\cdots+T^{u-1})
=\frac{T^{(lu+m)u-u+1}}{p^{2(lu+m)}}(1+T+\cdots+T^{u-1}).
\]

At this height $k=lu+m$ there are $p^{m-1}$ vertices.
The index $m$ ranges from $1$ to $\lfloor \frac {a+lu(v-u)} u \rfloor = \lfloor \frac a u \rfloor + l(v-u)$.
By the xylon formula (Theorem~\ref{th:xylon}), the Poincaré series is the sum of the $\Xi_{l,m}(T)$:
\begin{align*}
	S(T) 
	&= \sum_{l \ge 0} \ \sum_{m=1}^{l(v-u)+\lfloor \frac a u \rfloor}  p^{m-1} \Xi_{l,m}(T) \\
	&= \big(1+T+\cdots+T^{u-1}\big) \sum_{l\ge0} \  \sum_{m=1}^{l(v-u)+\lfloor \frac a u \rfloor} \frac{T^{mu -u+1}}{p^{m+1}} \frac{T^{lu^2}}{p^{2lu}} \\
	&= A(T) \sum_{l\ge1} \ \sum_{m=1}^{l(v-u)} \xi_{l,m}
	\ \ + \ \ A(T) \sum_{l\ge0} \ \sum_{m=l(v-u)+1}^{l(v-u)+\lfloor \frac a u \rfloor} \xi_{l,m}
	,
\end{align*}
where we have set $A(T) = 1+T+\cdots+T^{u-1}$ and
$\xi_{l,m} = \dfrac{T^{mu -u+1}}{p^{m+1}} \dfrac{T^{lu^2}}{p^{2lu}}$.

In the first summand, we write the Euclidean division $m = q(v-u)+r$, with here $1 \le r \le v-u$, and then interchange the sums:
\begin{align*}
	S_1(T) 
	&= A(T) \sum_{l \ge 1} \ \sum_{0\le q \le l-1} \ \sum_{1 \le  r \le v-u} \xi_{l,m} \\
	&= A(T) \sum_{q \ge 0} \ \sum_{l\ge q+1} \ \sum_{1 \le  r \le v-u} \xi_{l,m} \\
	& = A(T) \frac{T}{p^2}\sum_{q \ge 0}  \frac{T^{qu(v-u)}}{p^{q(v-u)}} \ \sum_{l\ge q+1} \frac{T^{lu^2}}{p^{2lu}} \ \sum_{1 \le  r \le v-u} \frac{T^{(r-1)u}}{p^{r-1}}  \\
	& = A(T) B(T)  \frac{T}{p^2} \sum_{q \ge 0}  \frac{T^{qu(v-u)}}{p^{q(v-u)}} \ \sum_{l\ge q+1} \frac{T^{lu^2}}{p^{2lu}}.
\end{align*}
Here we have set
\[
B(T) = 1+\frac{T^{u}}{p} + \left( \frac{T^{u}}{p} \right)^2 + \cdots + \left(  \frac{T^{u}}{p} \right) ^{v-u-1}.
\]

Thus:
\begin{align*}
	S_1(T) 
	&= A(T) B(T) \frac{T}{p^2}\sum_{q \ge 0}  \frac{T^{qu(v-u)}}{p^{q(v-u)}} 
	\frac{T^{(q+1)u^2}}{p^{2(q+1)u}}
	\ \sum_{l\ge 0} \left(\frac{T^{u^2}}{p^{2u}}\right)^l  \\
	&= A(T) B(T) \frac{T}{p^2} \left(1-\frac{T^{u^2}}{p^{2u}}\right)^{-1}
	\sum_{q \ge 0}  	
	\frac{T^{u^2}}{p^{2u}} 
	\left(  \frac{T^{uv}}{p^{u+v}} \right) ^q \\
	&= A(T) B(T) C(T) \left(1-\frac{T^{u^2}}{p^{2u}}\right)^{-1} \left(1-\frac{T^{uv}}{p^{u+v}}\right)^{-1},
\end{align*}
where we have set
\[
C(T) = \frac{T^{u^2+1}}{p^{2u+2}}.
\]

The second is, with $m = k +l(v-u)$:
\begin{align*}
	S_2(T) 
	&= A(T)   \sum_{l\ge0} \ \sum_{m=l(v-u)+1}^{l(v-u)+\lfloor \frac a u \rfloor} \xi_{l,m} \\
	&= A(T)   \sum_{l\ge0} \ \sum_{k=1}^{\lfloor \frac a u \rfloor} 
	\frac{T^{(k+l(v-u))u -u+1}}{p^{k+l(v-u)+1}} \dfrac{T^{lu^2}}{p^{2lu}} \\
	&= A(T)  \left( \sum_{k=1}^{\lfloor \frac a u \rfloor} \frac{T^{uk -u+1}}{p^{k+1}} \right)\ \sum_{l\ge0} 	 \frac{T^{luv}}{p^{l(u+v)}} \\
	&= A(T) D(T) \left(1-\frac{T^{uv}}{p^{u+v}}\right)^{-1},
\end{align*}
with the polynomial $D(T) = \sum_{k=1}^{\lfloor \frac a u \rfloor} \frac{T^{uk -u+1}}{p^{k+1}}$.

Conclusion: $S(T) = S_1(T) + S_2(T) \in \Qq(T)$, with denominators $p^{2u}-T^{u^2}$ and $p^{u+v}-T^{uv}$.
We have computed the series for a single outgoing branch; for all $p-1$ outgoing branches, the associated series is $(p-1)S(T) \in \Qq(T)$.

\subsection*{C. Series for terminal vertices, case $u$ does not divide $a$}

We now compute the series associated with the vertices in the case $a \% u \neq 0$.
Thus, we count the contribution of a terminal vertex of thickness $a_0=a \% u$ on each outgoing branch of the starred trunk.

A terminal vertex is characterized by a single integer $l$: its height is
$k = lu+l(v-u) + \lfloor \frac{a}{u} \rfloor + 1 = lv + \lfloor \frac{a}{u} \rfloor + 1$,
its thickness is $t=a_0 = a \% u$, and the value of the tree-top function is
$\varphi = (lv+\lfloor \frac{a}{u} \rfloor)u + a_0$.
For this vertex:
\[
\Xi_{l}(T) 
=\frac{T^{luv + \lfloor \frac{a}{u} \rfloor u + 1}}{p^{2(lv + \lfloor \frac{a}{u} \rfloor + 1)}}(1+T+\cdots+T^{a_0-1}).
\]
All the $p^{l(v-u) + \lfloor \frac{a}{u} \rfloor}$ terminal vertices (for one branch issued from the principal trunk among the $p-1$ possible ones) give rise to the following series, which is indeed rational:
\begin{align*}
	S(T) 
	&= \sum_{l\ge0} p^{l(v-u) + \lfloor \frac{a}{u} \rfloor} \; \Xi_{l}(T) \\
	&= \big(1+T+\cdots+T^{a_0-1}\big) \sum_{l\ge0} \frac{T^{luv+ \lfloor \frac{a}{u} \rfloor u +1}}{p^{2(lv + \lfloor \frac{a}{u} \rfloor+1)-(l(v-u)+ \lfloor \frac{a}{u} \rfloor)}} \\
	&= \big(1+T+\cdots+T^{a_0-1}\big) \frac{T^{\lfloor \frac{a}{u} \rfloor u +1}}{p^{\lfloor \frac{a}{u} \rfloor+2}} \sum_{l\ge0} \left( \frac{T^{uv}}{p^{u+v}} \right)^l \\
	&= \big( 1+T+\cdots+T^{a_0-1} \big)\frac{T^{\lfloor \frac{a}{u} \rfloor u +1}}{p^{\lfloor \frac{a}{u} \rfloor+2}} \left( 1-\frac{T^{uv}}{p^{u+v}}\right)^{-1}.
\end{align*}
The denominator is $p^{u+v}-T^{uv}$.

\subsection*{D. Series for the case $u$ divides $a$, with terminal vertex}

Here we use the notations of Section~\ref{ssec:alphazero}
and the results of Section \ref{ssec:alphapositive}. With these notations,
we consider the case $0 <\beta < \gamma$ (if $\beta=0$ there is no terminal vertex).
We have a computation very similar to the previous case.
We have seen that $\gamma$ vertices after leaving the principal stalk, the value of $\beta$ becomes constant.
We thus group the terminal vertices of thickness $\beta$ into packets of $p^{l(v-u)+ \lfloor \frac{a}{u} \rfloor-\gamma}$ terminal vertices (here we use that $\lfloor \frac{a}{u} \rfloor \ge \gamma$).
\begin{align*}
	S(T) 
	&= \sum_{l\ge0} p^{l(v-u)+ \lfloor \frac{a}{u} \rfloor-\gamma} \; \Xi_{l}(T) \\
	&= \big(1+T+\cdots+T^{\beta-1}\big) \sum_{l\ge0} \frac{T^{luv+ \lfloor \frac{a}{u} \rfloor u +1}}{p^{2(lv+ \lfloor \frac{a}{u} \rfloor+1)-l(v-u)- \lfloor \frac{a}{u} \rfloor+\gamma}} \\
	&= \big(1+T+\cdots+T^{\beta-1}\big) \frac{T^{\lfloor \frac{a}{u} \rfloor u+1}}{p^{\lfloor \frac{a}{u} \rfloor+\gamma+2}} \sum_{l\ge0} \left( \frac{T^{uv}}{p^{u+v}} \right)^l \\
	&= \big(1+T+\cdots+T^{\beta-1}\big)\frac{T^{\lfloor \frac{a}{u} \rfloor u +1}}{p^{\lfloor \frac{a}{u} \rfloor+\gamma+2}} \left( 1-\frac{T^{uv}}{p^{u+v}}\right)^{-1}.
\end{align*}

\subsection*{E. Series for the case $u$ divides $a$, with Hensel trees}

This occurs only when $u$ divides $a$. 
First, recall the generalization of Lemma~\ref{lem:Shensel} for a Hensel tree not attached to the root but to a vertex of height $k$ and with tree-top value $\varphi$ (with $n$ variables):
\[
S_{k,\varphi}(T) = \frac{T^\varphi}{p^{nk}} \left(1-\frac{T}{p}\right)^{-1}.
\]
Again we we use the notations of Section~\ref{ssec:alphazero}
and the results of Section \ref{ssec:alphapositive}.
The vertices where the Hensel trees are attached have height
$k = lu+l(v-u)+ \lfloor \frac{a}{u} \rfloor + 1 = lv+ \lfloor \frac{a}{u} \rfloor + 1$ at a vertex where the thickness is $t=\gamma$, and tree-top value $\varphi =  (lv+ \lfloor \frac{a}{u} \rfloor)u + \gamma$.
We have seen that, $\gamma$ vertices after leaving the principal stalk, the value of $\beta$ becomes constant.
There are $p^{l(v-u)+ \lfloor \frac{a}{u} \rfloor -\gamma}$ such vertices (for one branch issued from the principal trunk, here we use that $\lfloor \frac{a}{u} \rfloor \ge \gamma$), so the total contribution of these Hensel trees is:
\begin{align*}
	S(T) 
	&= \sum_{l \ge 0} p^{l(v-u)+ \lfloor \frac{a}{u} \rfloor-\gamma}  
	\frac{T^{luv+\lfloor \frac{a}{u} \rfloor u +\gamma}}
	{p^{2(lv+ \lfloor \frac{a}{u} \rfloor + 1)}} \left(1-\frac{T}{p}\right)^{-1}\\
	&= \left(1-\frac{T}{p}\right)^{-1} 
	\frac{T^{\lfloor \frac{a}{u} \rfloor u +\gamma}}
	{p^{\lfloor \frac{a}{u} \rfloor+\gamma+2}} 
	\sum_{l \ge 0}\left( \frac{T^{uv}}{p^{u+v}} \right)^l \\
	&= \left(1-\frac{T}{p}\right)^{-1} 
	\frac{T^{\lfloor \frac{a}{u} \rfloor u +\gamma}}
	{p^{\lfloor \frac{a}{u} \rfloor+\gamma+2}} 
	\left(1-\frac{T^{uv}}{p^{u+v}}\right)^{-1}.
\end{align*}
This is once again a rational fraction with denominators $p-T$ and $p^{u+v}-T^{uv}$.



\bibliographystyle{plain}
\bibliography{igusa.bib}

@article {BLQ2013,
	AUTHOR = {Berthomieu, J\'{e}r\'{e}my and Lecerf, Gr\'{e}goire and
	Quintin, Guillaume},
	TITLE = {Polynomial root finding over local rings and application to
	error correcting codes},
	JOURNAL = {Appl. Algebra Engrg. Comm. Comput.},
	FJOURNAL = {Applicable Algebra in Engineering, Communication and
	Computing},
	VOLUME = {24},
	YEAR = {2013},
	NUMBER = {6},
	PAGES = {413--443},
	ISSN = {0938-1279,1432-0622},
	MRCLASS = {13P05 (94B05)},
	MRNUMBER = {3128698},
	MRREVIEWER = {Hiram\ H.\ L\'{o}pez},
	DOI = {10.1007/s00200-013-0200-5},
	URL = {https://doi.org/10.1007/s00200-013-0200-5},
}

@article{BD26,
	author       = {Bodin, Arnaud and Drouin, Christian},
	title        = {Solutions of a polynomial equation modulo a prime power},
	journal      = {The American Mathematical Monthly},
	note         = {To appear (2026). Preprint available at \url{https://arxiv.org/abs/2510.07168}},
	year         = {2026},
	eprint       = {2510.07168},
	archivePrefix= {arXiv},
	primaryClass = {math.NT}
}

@article {bollaerts,
	AUTHOR = {Bollaerts, Dirk},
	TITLE = {On the {P}oincar\'e{} series associated to the {$p$}-adic
	points on a curve},
	JOURNAL = {Acta Arith.},
	FJOURNAL = {Polska Akademia Nauk. Instytut Matematyczny. Acta Arithmetica},
	VOLUME = {51},
	YEAR = {1988},
	NUMBER = {1},
	PAGES = {9--30},
	ISSN = {0065-1036},
	MRCLASS = {11S80 (11F85 11G20 14G10 14G20)},
	MRNUMBER = {959783},
	MRREVIEWER = {Hernando\ Enrique\ Sierra-Morales},
	DOI = {10.4064/aa-51-1-9-30},
	URL = {https://doi.org/10.4064/aa-51-1-9-30},
}

@book {borevich-shafarevich,
	AUTHOR = {Borevich, A. I. and Shafarevich, I. R.},
	TITLE = {Number theory},
	SERIES = {Pure and Applied Mathematics},
	VOLUME = {Vol. 20},
	NOTE = {Translated from the Russian by Newcomb Greenleaf},
	PUBLISHER = {Academic Press, New York-London},
	YEAR = {1966},
	PAGES = {x+435},
	MRCLASS = {10.00},
	MRNUMBER = {195803},
}

@mastersthesis{chakrabarti,
	author       = {Sayak Chakrabarti},
	title        = {Multivariate polynomials modulo prime powers: their roots, zeta-function and applications},
	school       = {Department of Computer Science and Engineering, IIT Kanpur},
	year         = {2022},
	type         = {Master's thesis},
	address      = {Kanpur, India}
}

@inproceedings {chakrabarti-saxena,
	AUTHOR = {Chakrabarti, Sayak and Saxena, Nitin},
	TITLE = {An effective description of the roots of bivariates {$\mod
	p^k$} and the related {I}gusa's local zeta function},
	BOOKTITLE = {Proceedings of the {I}nternational {S}ymposium on {S}ymbolic
	\& {A}lgebraic {C}omputation ({ISSAC} 2023)},
	PAGES = {135--144},
	PUBLISHER = {ACM, New York},
	YEAR = {2023},
	ISBN = {979-8-4007-0039-2},
	MRCLASS = {11D88 (68W30)},
	MRNUMBER = {4618421},
	DOI = {10.1145/3597066.3597115},
	URL = {https://doi.org/10.1145/3597066.3597115},
}

@incollection {denef,
	AUTHOR = {Denef, Jan},
	TITLE = {Report on {I}gusa's local zeta function},
	NOTE = {S\'eminaire Bourbaki, Vol.\ 1990/91},
	JOURNAL = {Ast\'erisque},
	FJOURNAL = {Ast\'erisque},
	NUMBER = {201-203},
	YEAR = {1991},
	PAGES = {Exp. No. 741, 359--386},
	ISSN = {0303-1179,2492-5926},
	MRCLASS = {11S40 (14G20)},
	MRNUMBER = {1157848},
	MRREVIEWER = {Boris\ Datskovsky},
}

@inproceedings {DS2020,
	AUTHOR = {Dwivedi, Ashish and Saxena, Nitin},
	TITLE = {Computing {I}gusa's local zeta function of univariates in
	deterministic polynomial-time},
	BOOKTITLE = {A{NTS} {XIV}---{P}roceedings of the {F}ourteenth {A}lgorithmic
	{N}umber {T}heory {S}ymposium},
	SERIES = {Open Book Ser.},
	VOLUME = {4},
	PAGES = {197--214},
	PUBLISHER = {Math. Sci. Publ., Berkeley, CA},
	YEAR = {2020},
	ISBN = {978-1-935107-08-8; 978-1-935107-07-1},
	MRCLASS = {11S40 (11Y16 14G10 68W30)},
	MRNUMBER = {4235114},
	MRREVIEWER = {Tommy\ Hofmann},
	DOI = {10.2140/obs.2020.4.197},
	URL = {https://doi.org/10.2140/obs.2020.4.197},
}

@incollection {DMS2019,
	AUTHOR = {Dwivedi, Ashish and Mittal, Rajat and Saxena, Nitin},
	TITLE = {Counting basic-irreducible factors mod {$p^k$} in
	deterministic poly-time and {$p$}-adic applications},
	BOOKTITLE = {34th {C}omputational {C}omplexity {C}onference},
	SERIES = {LIPIcs. Leibniz Int. Proc. Inform.},
	VOLUME = {137},
	PAGES = {Art. No. 15, 29},
	PUBLISHER = {Schloss Dagstuhl. Leibniz-Zent. Inform., Wadern},
	YEAR = {2019},
	ISBN = {978-3-95977-116-0},
	MRCLASS = {68R05 (11T06)},
	MRNUMBER = {3984620},
}

@inproceedings {DMS2019P4,
    AUTHOR = {Dwivedi, Ashish and Mittal, Rajat and Saxena, Nitin},
     TITLE = {Efficiently factoring polynomials modulo {$p^4$}},
 BOOKTITLE = {I{SSAC}'19---{P}roceedings of the 2019 {ACM} {I}nternational
              {S}ymposium on {S}ymbolic and {A}lgebraic {C}omputation},
     PAGES = {139--146},
 PUBLISHER = {ACM, New York},
      YEAR = {2019},
      ISBN = {978-1-4503-6084-5},
   MRCLASS = {11T06 (11Y05 13P05 68W30)},
  MRNUMBER = {4007453},
       DOI = {10.1145/3326229.3326233},
       URL = {https://doi.org/10.1145/3326229.3326233},
}

@article {hayes-nutt,
	AUTHOR = {Hayes, David R. and Nutt, Michael D.},
	TITLE = {Reflective functions on {$p$}-adic fields},
	JOURNAL = {Acta Arith.},
	FJOURNAL = {Polska Akademia Nauk. Instytut Matematyczny. Acta Arithmetica},
	VOLUME = {40},
	YEAR = {1981/82},
	NUMBER = {3},
	PAGES = {229--248},
	ISSN = {0065-1036},
	MRCLASS = {12B40},
	MRNUMBER = {664613},
	MRREVIEWER = {L.\ Corwin},
	DOI = {10.4064/aa-40-3-229-248},
	URL = {https://doi.org/10.4064/aa-40-3-229-248},
}

@article {igusa,
	AUTHOR = {Igusa, Jun-ichi},
	TITLE = {Complex powers and asymptotic expansions. {I}. {F}unctions of
	certain types},
	JOURNAL = {J. Reine Angew. Math.},
	FJOURNAL = {Journal f\"ur die Reine und Angewandte Mathematik. [Crelle's
	Journal]},
	VOLUME = {268/269},
	YEAR = {1974},
	PAGES = {110--130},
	ISSN = {0075-4102,1435-5345},
	MRCLASS = {10H25 (12B40)},
	MRNUMBER = {347753},
	MRREVIEWER = {Stephen\ Haris},
	DOI = {10.1515/crll.1974.268-269.110},
	URL = {https://doi.org/10.1515/crll.1974.268-269.110},
}

@article {Kopp,
	AUTHOR = {Kopp, Leann and Randall, Natalie and Rojas, J. Maurice and
	Zhu, Yuyu},
	TITLE = {Randomized polynomial-time root counting in prime power rings},
	JOURNAL = {Math. Comp.},
	FJOURNAL = {Mathematics of Computation},
	VOLUME = {89},
	YEAR = {2020},
	NUMBER = {321},
	PAGES = {373--385},
	ISSN = {0025-5718,1088-6842},
	MRCLASS = {11Y16 (11T06 11Y99 16P10)},
	MRNUMBER = {4011547},
	MRREVIEWER = {Steven\ D.\ Galbraith},
	DOI = {10.1090/mcom/3431},
	URL = {https://doi.org/10.1090/mcom/3431},
}

@incollection {PoVe,
	AUTHOR = {Potemans, Naud and Veys, Willem},
	TITLE = {Introduction to {$p$}-adic {I}gusa zeta functions},
	BOOKTITLE = {{$p$}-adic analysis, arithmetic and singularities},
	SERIES = {Contemp. Math.},
	VOLUME = {778},
	PAGES = {71--102},
	PUBLISHER = {Amer. Math. Soc.},
	YEAR = {2022},
	ISBN = {978-1-4704-6779-1},
	MRCLASS = {11S40 (11M41 11S80 14E15 14G10 14G20 14H20)},
	MRNUMBER = {4419243},
	MRREVIEWER = {David\ Villa-Hern\'andez},
	DOI = {10.1090/conm/778/15655},
	URL = {https://doi.org/10.1090/conm/778/15655},
}

@article {ScSt,
	AUTHOR = {Schmidt, Wolfgang M. and Stewart, C. L.},
	TITLE = {Congruences, trees, and {$p$}-adic integers},
	JOURNAL = {Trans. Amer. Math. Soc.},
	FJOURNAL = {Transactions of the American Mathematical Society},
	VOLUME = {349},
	YEAR = {1997},
	NUMBER = {2},
	PAGES = {605--639},
	ISSN = {0002-9947,1088-6850},
	MRCLASS = {11D88 (11A07 11S05)},
	MRNUMBER = {1340185},
	MRREVIEWER = {Jeffrey\ Lin\ Thunder},
	DOI = {10.1090/S0002-9947-97-01547-X},
	URL = {https://doi.org/10.1090/S0002-9947-97-01547-X},
}

@article {ZG2003,
	AUTHOR = {Zuniga-Galindo, W. A.},
	TITLE = {Computing {I}gusa's local zeta functions of univariate
	polynomials, and linear feedback shift registers},
	JOURNAL = {J. Integer Seq.},
	FJOURNAL = {Journal of Integer Sequences},
	VOLUME = {6},
	YEAR = {2003},
	NUMBER = {3},
	PAGES = {Article 03.3.6, 18},
	ISSN = {1530-7638},
	MRCLASS = {11S40 (94A55)},
	MRNUMBER = {2046406},
	MRREVIEWER = {Margaret\ M.\ Robinson},
}

\appendix

\newpage

\section{Example $P(x,y) = x^2-y^3$}
\label{sec:example}

The goal of this appendix is to illustrate the results of the previous sections with the example
$P(x,y) = x^2-y^3$, where $p>2$ is a prime number.
In particular, we compute the trunk of $P$, its Poincaré series, and the first values $N_e$ for the number of solutions of
$x^2-y^3 \equiv 0 \pmod{p^e}$.
We recover the same results as those obtained by Denef’s formula \cite{PoVe}.

We follow step by step the method of Section~\ref{sec:rational}, which will allow us to make the series $S(T)$ completely explicit.

\subsection{Preliminaries}

Let us make the following remarks, which we will need later:
\begin{itemize}
	\item For every prime $p$, the congruence $x^2-y^3 \equiv 0 \pmod{p}$ has $p$ solutions.
	They are obtained via the parametrization $(t^3,t^2)$ for $t\in \llbracket 0,p-1\rrbracket$.
	
	\item It is known that there are $\frac{p-1}{2}$ elements $y \in \Zz/p\Zz \setminus\{0\}$ that are squares modulo $p$.
	
	\item Finally, $y^3$ is a square modulo $p$ if and only if $y$ is.
	A proof using the Legendre symbol is:
	$(\frac{y^3}{p}) = +1 \iff (\frac{y}{p})^3 = +1
	\iff (\frac{y}{p}) = +1$.
\end{itemize}

As in Section~\ref{sec:rational}, we compute the trunk (see the figures on the following pages) and we partition this trunk with an explicit computation of the series associated with each piece.

\subsection{Principal stalk}

Above $(0,0)$, there is an infinite stalk whose vertices have thickness $t=2$.
We denote by $P_k(x,y) = x^2-p^k y^3$ the $k$-th successor of $P$ for the root $(0,0)$.

Let us compute the Poincaré series $S_A(T)$ associated with this stalk via the xylon formula~\ref{th:xylon}.
The vertices are indexed by their height $k \ge 0$.
For $k>0$ the tree-top value is $\varphi = 2k$, and the thickness is $t=2$.
We also refer to Lemma~\ref{lem:Sstalk}.
Thus:
\begin{align*}
	S_A(T) 
	&= 1+\sum_{k > 0} \Xi_k(T) 
	= 1 \  + \ \sum_{k>0} \frac{T^{2k-1}}{p^{2k}} (1+T) \\
	&=  1 \  + \ ( 1+T) \frac{T}{p^2} \sum_{k \ge 0} \left( \frac{T^2}{p^2} \right)^k   \\
	&=  1 \  + \ ( 1+T) \frac{T}{p^2} \left( 1-\frac{T^2}{p^2} \right)^{-1}   \\ 
	&= \frac{p^2+T}{p^2-T^2}
\end{align*}

\subsection{Other branches from the root}

We now consider the nonzero roots of $P(x,y) \equiv 0 \pmod{p}$.
We have seen that there are $p-1$ nonzero solutions $(r_i,s_i)$ (their exact values depend on $p$, but not their number).
These branches must be studied separately since, in the language of Section~\ref{sec:rational}, they are not ``sufficiently high on the trunk''.
For such a nonzero root $(r,s)$, we have:
\[
P(r+px, s+py) = (r+px)^2 - (s+py)^3 = r^2-s^3 + 2rpx - 3s^2py + \cdots
\]
We know that $p \mid r^2-s^3$.
For $p>2$, we have $p^2 \notdivides 2rp$, hence the thickness is $t=1$.
The successor is
\[
Q_1(x,y) = \frac{r^2-s^3}{p} + 2rx -3s^2y + \cdots.
\]
Since the coefficient of $x$ is nonzero modulo $p$, we are in the situation of Hensel’s lemma.
Conclusion: to each of the $p-1$ nonzero roots of $\reduc{P}$, we associate an outgoing branch from the root, which is in fact a Hensel tree.

Each of these trees is attached at height $k_0=1$ and with tree-top value $\varphi_0=1$.
We apply $p-1$ times the shifted version of Lemma~\ref{lem:Shensel}:

\[
S_B(T) 
= (p-1)  \frac{T^{\varphi_0}}{p^{2k_0}} \left(1-\frac{T}{p}\right)^{-1}
= \frac{(p-1)T}{p(p-T)}.
\]

Remark: the notation for the series $S_B(T), S_C(T), \ldots$ is slightly different here from that of Section~\ref{sec:rational}.

\subsection{Branches from an odd height}

See Figure~\ref{fig:odd}.
We start with height $1$ and $P_1(x,y)=x^2-py^3$: the nonzero roots of
$\reduc{P_1}(x,y) \equiv  x^2 \pmod{p}$ are the $(0,s)$ with
$s\in\llbracket1,p-1\rrbracket$.
For one of these roots, we have
$P_1(0+px,s+py) = p Q_1(x,y)$ with
$Q_1(x,y) = px^2 - (s+py)^3$.
Thus each root has thickness $t=1$. But then
$\reduc{Q_1}(x,y) \equiv s \not\equiv 0 \pmod{p}$, and the trunk stops there.
Conclusion: the branches (other than the principal stalk) issued from the vertex of height $1$
consist of $p-1$ terminal vertices of thickness $1$.

\bigskip

More generally, for a vertex of height $k=2l+1$, the polynomial is
$P_k(x,y) = x^2-p^ky^3$.
We always exclude the root $(0,0)$.
The structure of the outgoing branches is as follows:
\begin{itemize}
	\item $p-1$ outgoing branches from the vertex of height $k$
	(there would be $p$ in total if we added the branch of the principal stalk),
	which lead to $p-1$ vertices of thickness $2$.
	
	\item From each of these vertices are attached $p$ vertices of thickness $t=2$,
	to which are attached $p$ vertices of thickness $t=2$, and so on,
	up to height $k+l$.
	
	\item From each of the vertices of height $k+l$ are attached $p$ vertices of thickness $t=1$,
	and the trunk stops there.
\end{itemize}

\bigskip

We call the \defi{partial trunk} the sub-trunk consisting of the vertices of thickness $2$,
and we then study the contribution of the terminal vertices (of thickness $1$).

We compute the contribution of the partial trunk (odd case) for each of the $p-1$ outgoing branches.
A vertex on such a branch is indexed by $(l,m)$, where $2l+1$ is the height of the attachment point
on the principal stalk and $m$ is the distance from this attachment point (with $1\le m \le l$).
Thus the height of the vertex is $k = (2l+1)+m$, and $\varphi = 2k$ with $t=2$.
For a fixed index $m$ there are $p^{m-1}$ vertices.
\begin{align*}
	\tilde{S}_C(T) 
	&= \sum_{l>0}\sum_{m=1}^{l} p^{m-1} \Xi_{l,m}(T) 
	= \sum_{l>0}\sum_{m=1}^{l} p^{m-1} \frac{T^{\varphi-t+1}}{p^{2k}}(1+T) \\
	&= (1+T)\sum_{l \ge 1}\sum_{m=1}^{l} \frac{T^{4l}}{p^{4l}} \frac{T^{2m+1}}{p^{m+3}} \\
	&= (1+T)\sum_{m \ge 1}\sum_{l \ge m} \frac{T^{4l}}{p^{4l}} \frac{T^{2m+1}}{p^{m+3}} \\
	&= (1+T)\sum_{m \ge 1}\frac{T^{2m+1}}{p^{m+3}}\frac{T^{4m}}{p^{4m}} \sum_{l\ge0} \frac{T^{4l}}{p^{4l}} \\
	&= (1+T)\sum_{m \ge 1}\frac{T^{6m+1}}{p^{5m+3}} \left(1-\frac{T^{4}}{p^{4}}\right)^{-1}	\\
	&= (1+T)\frac{T^7}{p^8} \left(1-\frac{T^{6}}{p^{5}}\right)^{-1}\left(1-\frac{T^{4}}{p^{4}}\right)^{-1} \\
	&= \frac{p(1+T)T^7}{(p^4-T^4)(p^5-T^6)}.
\end{align*}
Summing over the $p-1$ branches gives:
\[
S_C(T) = (p-1)\tilde{S}_C(T) = \frac{p(p-1)(1+T)T^7}{(p^4-T^4)(p^5-T^6)}.
\]

It remains to compute the contribution of the terminal vertices.
There are $p^l$ of them, at height $k=3l+2$, with $\varphi=6l+3$ and $t=1$.
\begin{align*}
	\tilde{S}_D(T) 
	&= \sum_{l \ge 0} p^l \Xi_l(T) 
	= \sum_{l \ge 0} p^l \frac{T^{\varphi-t+1}}{p^{2k}}(1) \\
	&= \sum_{l \ge 0} \frac{T^{6l+3}}{p^{5l+4}}
	= \frac{T^3}{p^4}\left(1-\frac{T^{6}}{p^{5}}\right)^{-1}.
\end{align*}
Summing over the $p-1$ branches gives:
\[
S_D(T) = (p-1)\tilde{S}_D(T) = \frac{p(p-1)T^3}{p^5-T^6}.
\]

\subsection{Branches from an even height}

See Figure~\ref{fig:even}.
We consider a vertex of the principal stalk of height $k=2l$ with $l>0$.

Let us start with height $2$ and $P_2(x,y) = x^2-p^2y^3$.
The roots modulo $p$, other than $(0,0)$, are the $(0,s)$ with
$s \in \llbracket 1, p-1 \rrbracket$.
Each root $(0,s)$ has thickness $t=2$ and has successor
$Q_1(x,y) = x^2-(s+py)^3$.

We now focus on $Q_1(x,y)$, with $s$ fixed.
Since $\reduc{Q_1}(x,y) \equiv x^2-s^3 \pmod{p}$, we discuss according to whether $s^3$ is a square modulo $p$:
\begin{itemize}
	\item if $s$ is not a square modulo $p$, then neither is $s^3$, and thus
	$\reduc{Q_1}(x,y) \equiv 0 \pmod{p}$ has no solutions.
	The trunk stops there.
	This happens for $\frac{p-1}{2}$ nonzero values of $s$ that are not squares.
	
	\item if $s \neq 0$ is a square modulo $p$, then $s^3$ is also a square.
	In this case $\reduc{Q_1}(x,y) \equiv 0 \pmod{p}$ has $2p$ solutions
	$(r'_1,s')$ and $(r'_2,s')$, with two possible values for the first coordinate
	($r'_2=-r'_1$) and $s' \in \llbracket 0,p-1 \rrbracket$ arbitrary.
	Fix such a solution $(r',s')$.
	Then
	\[
	Q_1(r'+px,s'+py) = r'^2-s^3 + 2pr'x + \cdots
	\]
	This is therefore a vertex of thickness $t=1$, and indeed a Hensel configuration.
	Conclusion in this case: from the principal stalk there branch off $\frac{p-1}{2}$ vertices of thickness $t=2$.
	To each of these vertices are attached $2p$ Hensel trees.
\end{itemize}
Conclusion: there are $p-1$ outgoing branches from the vertex of height $2$:
half lead to terminal vertices of thickness $2$; the other half lead to vertices of thickness $2$
to which, for each such vertex, $2p$ Hensel trees are attached.

\medskip

Remark: for this example there are no vertices of Case D of Section~\ref{sec:rational}.
Indeed we assumed $u=2$ and $p>2$, hence $\val_p(pu)=1$.
Thus for $\gamma = \val_p(pu)$, there is no $\beta$ satisfying $0 < \beta < \gamma$.

\bigskip

More generally, for a vertex of the principal stalk of height $k=2l$ (with $l>0$),
the polynomial is $P_k(x,y) = x^2-p^ky^3$, and we exclude the root $(0,0)$ from the discussion.
We denote by $(0,s)$ with $s \in \llbracket 1 , p-1 \rrbracket$ the roots of $\reduc{P_k}$.
The structure of the outgoing branches is as follows:
\begin{itemize}
	\item $p-1$ outgoing branches from the vertex of height $k$
	(there would be $p$ in total if we counted the branch of the principal stalk),
	which lead to $p-1$ vertices of thickness $2$.
	
	\item From each of these vertices are attached $p$ vertices with $t=2$,
	to which are attached $p$ vertices with $t=2$, and so on,
	up to height $k+l$.
	
	\item For the end of the partial trunk:
	\begin{itemize}
		\item for half of these vertices of height $k+l$, the partial trunk stops there:
		these are the vertices arising from a branch where $s$ is not a square modulo $p$,
		
		\item for the other half, to each of these vertices of height $k+l$ are attached $2p$ Hensel trees:
		these are the vertices arising from a branch where $s$ is a square modulo $p$.
	\end{itemize}
\end{itemize}

\bigskip

The contribution of the partial trunk (even case) for the $p-1$ outgoing branches is very similar to the odd case.
A vertex of the partial trunk is indexed by $(l,m)$ with height $k = 2l+m$, $\varphi = 2k$, $t=2$.
For a fixed index $m$, there are $p^{m-1}$ vertices.
\begin{align*}
	S_E(T) 
	&= (p-1)\sum_{l>0}\sum_{m=1}^{l} p^{m-1} \Xi_{l,m}(T) \\
	&= \cdots \\
	&= (p-1)(1+T) \frac{T^5}{p^6}\left(1-\frac{T^{4}}{p^{4}}\right)^{-1} \left(1-\frac{T^{6}}{p^{5}}\right)^{-1}\\
	&= \frac{p^3(p-1)(1+T)T^5}{(p^4-T^4)(p^5-T^6)}.
\end{align*}

It remains to compute the contribution of the Hensel trees, for the $\frac{p-1}{2}$ outgoing branches associated with an $s$
that is a square modulo $p$ (for the other branches the trunk is limited to the partial trunk).

For the height $k=2l$ of the attachment vertex, there are thus $\frac{p-1}{2}$ outgoing branches.
For each of these branches, there are $p^{l-1}$ vertices at the end of the partial trunk, and at each such vertex $2p$ Hensel trees are attached.
These Hensel trees are attached at height $k_0 = 3l+1$ with tree-top value $\varphi_0 = 6l+1$.
Thus:
\begin{align*}
	S_F(T) 
	&= \frac{p-1}{2} \sum_{l>0} p^{l-1} (2p)  \frac{T^{\varphi_0}}{p^{2k_0}} \left(1-\frac{T}{p}\right)^{-1}\\
	&= (p-1) \left(1-\frac{T}{p}\right)^{-1} \sum_{l>0}\frac{T^{6l+1}}{p^{5l+2}} \\
	&= (p-1) \left(1-\frac{T}{p}\right)^{-1} \frac{T^7}{p^7} \left(1-\frac{T^6}{p^5}\right)^{-1} \\
	&= \frac{(p-1)T^7}{p(p-T)(p^5-T^6)}.
\end{align*}

\subsection{The trunk of $x^2-y^3$}

Before drawing the trunk, let us first recall the trunk of a Hensel tree.
In what follows, we will symbolize such a Hensel tree by the notation \circledtextH.

\begin{figure}[H]
	\myfigure{1.0}{
\begin{tikzpicture}[scale=2]

\tikzset{
  line/.style = {
  },
  vector/.style = {
    thick,-latex
  },
  dot/.style = {
    insert path={
      node[scale=3]{.}
    }
  }
}

\def\p{3}
 \path
   (0,0) coordinate (O)
 ;

\foreach \i in {0,...,2} {
     \path (\i-1,1) coordinate (P\i1);
  }

\foreach \i in {0,...,8} {
     \path ({(\i-4)/3},2) coordinate (P\i2);
  }
\foreach \i in {0,...,26} {
     \path ({(\i-13)/9},3) coordinate (P\i3);
  }

  \foreach \i in {0,1,2} {
     \path (O) edge[line] (P\i1);
  }
  \foreach \i in {0,1,2} {
     \foreach \k in {0,1,2} {
         \pgfmathtruncatemacro\ii{\i+3*\k}
          \path (P\k1) edge[line] (P\ii2);
     }
  }

  \foreach \i in {0,1,2} {
     \foreach \k in {0,...,8} {
         \pgfmathtruncatemacro\ii{\i+3*\k}
          \path (P\k2) edge[line,thin,dashed,] (P\ii3);
     }
  }

 \path
   (O) [dot] node{};  
 ;
  \foreach \i in {0,...,2} {
     \path  (P\i1) [dot] {}; 
  }
  \foreach \i in {0,...,8} {
     \path  (P\i2) [dot] {};
  }
%
%

  \path (2,1) node{$\# = p$};
  \path (2,2) node{$\# = p^2$};


\end{tikzpicture}%

	}
	\caption{A Hensel tree for $n=2$ variables, drawn here for $p=3$.
		From each vertex there are $p$ outgoing branches. All vertices have thickness $1$.}
	\label{fig:hensel}
\end{figure}
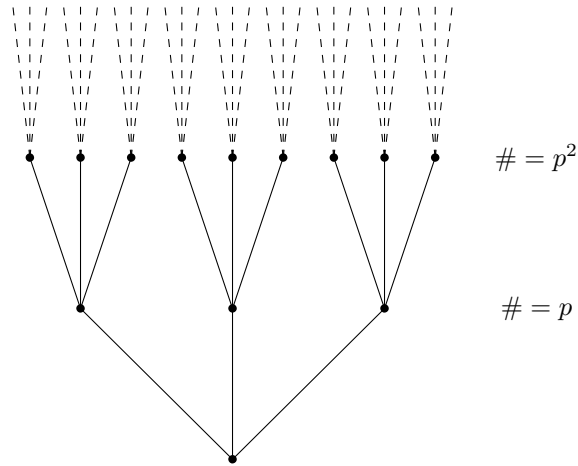

\medskip

We now depict the trunk of $P_0(x,y)=x^2-y^3$ up to height $3$.
For a terminal vertex of the trunk (that is, one having no successor),
we sometimes add the symbol $\varnothing$ above the vertex in order to clarify the diagram.

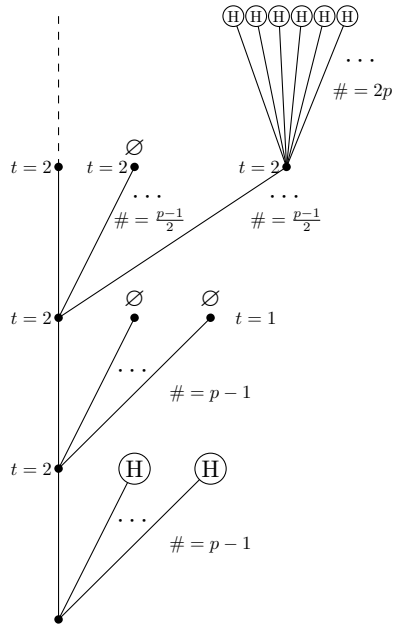
\begin{figure}[H]
	\myfigure{1.0}{
\begin{tikzpicture}[scale=2]

\tikzset{
  line/.style = {
  },
  vector/.style = {
    thick,-latex
  },
  dot/.style = {
    insert path={
      node[scale=3]{.}
    }
  }
}

 \path
   (0,0) coordinate (P0)
	(0,1) coordinate (P1)
	(0,2) coordinate (P2)
	 (0,3) coordinate (P3)
	 (0,4) coordinate (P4)
	(0,4.7) coordinate (P5)
 ;

\path 
	  (P0) edge[line] (P1)
	  (P1) edge[line] (P2)
	  (P2) edge[line] (P3)
	  (P3) edge[line, dashed] (P4)
;

 \path
   (P0) [dot] node{}
	(P1) [dot] node[left, scale=0.7]{$t=2$}
	(P2) [dot] node[left, scale=0.7]{$t=2$}
	(P3) [dot] node[left, scale=0.7]{$t=2$}
 ;

 \path
   (0.5,1) coordinate (P01)
	(1,1) coordinate (P02)
 ;

  \foreach \i in {1,2} {
       \path (P0) edge[line] (P0\i);
		\path (P0\i) node[draw,circle,inner sep=1pt,fill=white, scale=0.9]{H};
  }

\node at (0.5,0.65) [scale=1]{$\cdots$};
\path (1.0,0.5) node[scale=0.7]{$\#=p-1$}; 
\path (1.3,2) node[scale=0.7]{$t=1$}; 

 \path
   (0.5,2) coordinate (P11)
	(1,2) coordinate (P12)
 ;

  \foreach \i in {1,2} {
       \path (P1) edge[line] (P1\i);
		\path (P1\i) [dot] node[above]{$\varnothing$};
  }
\node at (0.5,1.65) [scale=1]{$\cdots$};
\path (1.0,1.5) node[scale=0.7]{$\#=p-1$};

 \path
   (0.5,3) coordinate (P31)
	(1.5,3) coordinate (P32)
 ;

  \foreach \i in {1,2} {
       \path (P2) edge[line] (P3\i);

  }
		\path (P31) [dot] node[above]{$\varnothing$} node[left,scale=0.7]{$t=2$};
		\path (P32) [dot] node[left,scale=0.7]{$t=2$};

\node at (0.6,2.8) [scale=1]{$\cdots$};
\node at (1.5,2.8) [scale=1]{$\cdots$};
\path (0.6,2.65) node[scale=0.7]{$\#=\frac{p-1}{2}$}; 
\path (1.5,2.65) node[scale=0.7]{$\#=\frac{p-1}{2}$};


\foreach \i in {1,2,3,4,5,6} {
    \path   (1.0+0.15*\i,4) coordinate (P4\i);
	 \path (P32) edge[line] (P4\i);
	\path (P4\i) node[draw,circle,inner sep=1pt,fill=white, scale=0.6]{H};
}

\node at (2.0,3.7) [scale=1]{$\cdots$};
\path (2.0,3.5) node[scale=0.7]{$\#=2p$};

\end{tikzpicture}%

	}
	\caption{The trunk up to height $3$. The drawing is for $p=3$.
		The dotted parts and the numbers of branches correspond to the cases $p>3$.}
	\label{fig:stalk}
\end{figure}

\medskip

To obtain the complete trunk, one must understand the branches attached to the principal stalk.
We start by representing the branches attached to a vertex of odd height.

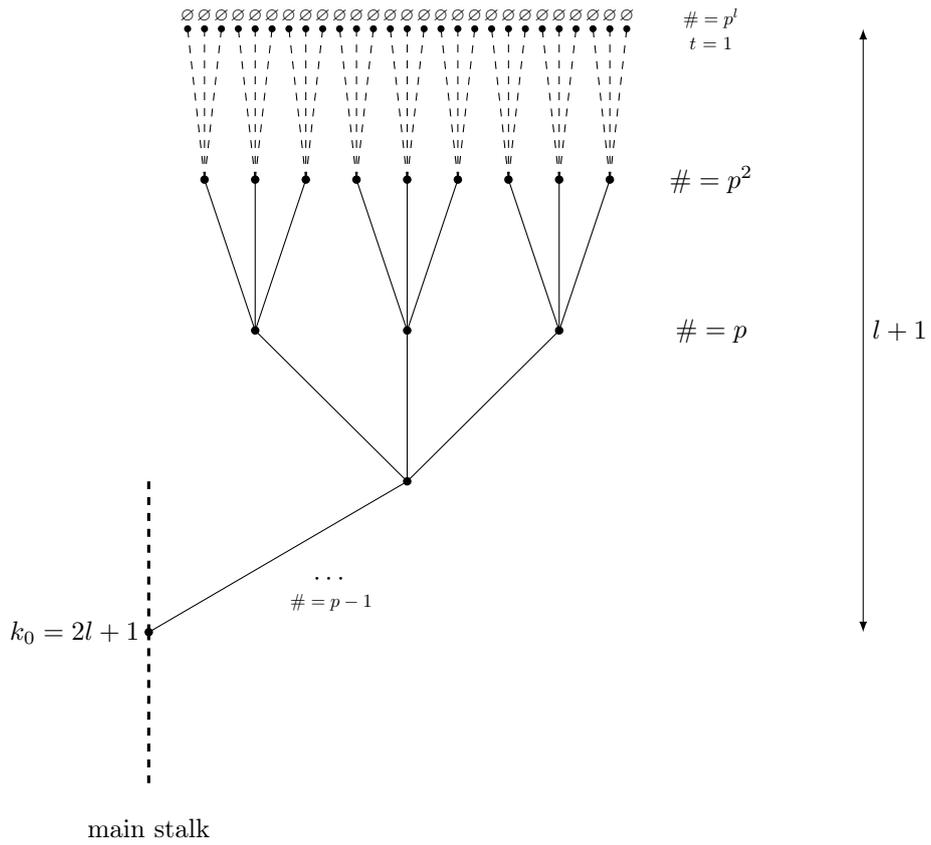
\begin{figure}[H]
	\myfigure{1.0}{
\begin{tikzpicture}[scale=2]

\tikzset{
  line/.style = {
  },
  vector/.style = {
    thick,-latex
  },
  dot/.style = {
    insert path={
      node[scale=3]{.}
    }
  }
}

\def\p{3}
 \path
   (0,0) coordinate (O)
   (-1.7,-1) coordinate (Q)
 ;

\foreach \i in {0,...,2} {
     \path (\i-1,1) coordinate (P\i1);
  }

\foreach \i in {0,...,8} {
     \path ({(\i-4)/3},2) coordinate (P\i2);
  }
\foreach \i in {0,...,26} {
     \path ({(\i-13)/9},3) coordinate (P\i3);
  }

\path (Q) edge[line] (O);
\path (Q) edge[line,very thick,dashed] ++(0,1);
\path (Q) edge[line,very thick,dashed] ++(0,-1);
\path  (Q) [dot] node[left]{$k_0=2l+1$};
\node at (-1.7,-2.3) {main stalk};
\node at (-0.5,-0.65) [scale=1]{$\cdots$};
\path (-0.5,-0.8) node[scale=0.7]{$\#=p-1$}; 
 
  \foreach \i in {0,1,2} {
     \path (O) edge[line] (P\i1);
  }
  \foreach \i in {0,1,2} {
     \foreach \k in {0,1,2} {
         \pgfmathtruncatemacro\ii{\i+3*\k}
          \path (P\k1) edge[line] (P\ii2);
     }
  }

  \foreach \i in {0,1,2} {
     \foreach \k in {0,...,8} {
         \pgfmathtruncatemacro\ii{\i+3*\k}
          \path (P\k2) edge[line,thin,dashed,] (P\ii3);
			\path (P\ii3) node[above,scale=0.7]{$\varnothing$};
     }
  }

 \path
   (O) [dot] node{};  
 ;
  \foreach \i in {0,...,2} {
     \path  (P\i1) [dot] {}; 
  }
  \foreach \i in {0,...,8} {
     \path  (P\i2) [dot] {};
  }
  \foreach \i in {0,...,26} {
     \path  (P\i3) node[scale=2.5]{.};
  }
%
%

  \path (2,1) node{$\# = p$};
  \path (2,2) node{$\# = p^2$};
  \path (2,3) node[align=center, text width=3em, scale=0.7]{$\# = p^{l}$ $t=1$};

  \draw[<->,>=latex] (3,-1) -- ++(0,4) node[midway, right]{$l+1$};

\end{tikzpicture}%

	}
	\caption{A branch attached to a vertex of odd height $2l+1$ of the principal stalk.
		There are $p-1$ such outgoing branches from this vertex.
		The branches are finite and stop at height $3l+2$.
		All vertices have thickness $2$, except the terminal vertices, which have thickness $1$.}
	\label{fig:odd}
\end{figure}

\medskip

The branches attached to a vertex of even height split into two cases,
according to whether the corresponding integer is a square modulo $p$ or not.

\begin{figure}[H]
	\myfigure{1.0}{
\begin{tikzpicture}[scale=2]

\tikzset{
  line/.style = {
  },
  vector/.style = {
    thick,-latex
  },
  dot/.style = {
    insert path={
      node[scale=3]{.}
    }
  }
}

\def\p{3}
 \path
   (0,0) coordinate (O)
   (-1.8,-1) coordinate (Q)
 ;

\foreach \i in {0,...,2} {
     \path (\i-1,1) coordinate (P\i1);
  }

\foreach \i in {0,...,8} {
     \path ({(\i-4)/3},2) coordinate (P\i2);
  }
\foreach \i in {0,...,26} {
     \path ({(\i-13)/9},3) coordinate (P\i3);
 }

\path   (-3.6,0) coordinate (OO);
\foreach \i in {0,...,2} {
     \path (-3.6+\i-1,1) coordinate (PP\i1);
  }

\foreach \i in {0,...,8} {
     \path ({-3.6+(\i-4)/3},2) coordinate (PP\i2);
  }
\foreach \i in {0,...,26} {
     \path ({-3.6+(\i-13)/9},3) coordinate (PP\i3);
 }

\path (Q) edge[line] (O);
\path (Q) edge[line,very thick,dashed] ++(0,1);
\path (Q) edge[line,very thick,dashed] ++(0,-1);
\path  (Q) [dot] node[below left]{$k_0=2l$};
\node at (-1.8,-2.3) {main stalk};
\node at (-0.75,-0.65) [scale=1]{$\cdots$};
\path (-0.75,-0.8) node[scale=0.7]{$\#=\frac{p-1}{2}$}; 

\path (Q) edge[line] (OO);
\node at (-2.5-0.5,-0.65) [scale=1]{$\cdots$};
\path (-2.5-0.5,-0.8) node[scale=0.7]{$\#=\frac{p-1}{2}$};

  \foreach \i in {0,1,2} {
     \path (O) edge[line] (P\i1);
  }
  \foreach \i in {0,1,2} {
     \foreach \k in {0,1,2} {
         \pgfmathtruncatemacro\ii{\i+3*\k}
         \path (P\k1) edge[line, dashed] (P\ii2);
			\path (P\ii2) node[above,scale=0.7]{$\varnothing$};
     }
  }


 \path
   (O) [dot] node{};  
 ;
  \foreach \i in {0,...,2} {
     \path  (P\i1) [dot] {}; 
  }
  \foreach \i in {0,...,8} {
     \path  (P\i2) [dot] {};
  }

  \path (1.8,1) node[scale=0.7]{$\# = p$};
  \path (1.8,2) node[scale=0.7]{$\# = p^{l-1}$};
  \path (1.8,1.6) node[scale=1.0]{$\vdots$};

  \draw[<->,>=latex] (2.2,-1) -- ++(0,3) node[midway, right]{$l$};

  \foreach \i in {0,1,2} {
     \path (OO) edge[line] (PP\i1);
  }
  \foreach \i in {0,1,2} {
     \foreach \k in {0,1,2} {
         \pgfmathtruncatemacro\ii{\i+3*\k}
          \path (PP\k1) edge[line, dashed] (PP\ii2);
     }
  }

  \foreach \i in {0,...,8} {
          \draw (PP\i2) -- ++(0,0.25) node[draw,circle,inner sep=1pt,fill=white, scale=0.5]{H} node[below right,scale=0.45]{$\times 2p$};
    
   }

  \foreach \i in {0,...,2} {
     \path  (PP\i1) [dot] {}; 
  }
  \foreach \i in {0,...,8} {
     \path  (PP\i2) [dot] {};
  }

\node at (-3.6,-1.5) {case of squares modulo $p$};
\node at (0,-1.5) {case of non-squares modulo $p$};
\end{tikzpicture}%

	}
	\caption{Two branches attached to a vertex of even height $2l$ (with $l>0$) of the principal stalk.
		On the right, there are $\frac{p-1}{2}$ finite outgoing branches up to height $3l$,
		corresponding to integers that are not squares modulo $p$; all their vertices have thickness $2$.
		On the left, there are also $\frac{p-1}{2}$ outgoing branches corresponding to nonzero integers
		that are squares modulo $p$; the vertices up to height $3l$ have thickness $2$,
		and on each of the $p^{l-1}$ vertices at height $3l$ are attached $2p$ Hensel trees.}
	\label{fig:even}
\end{figure}
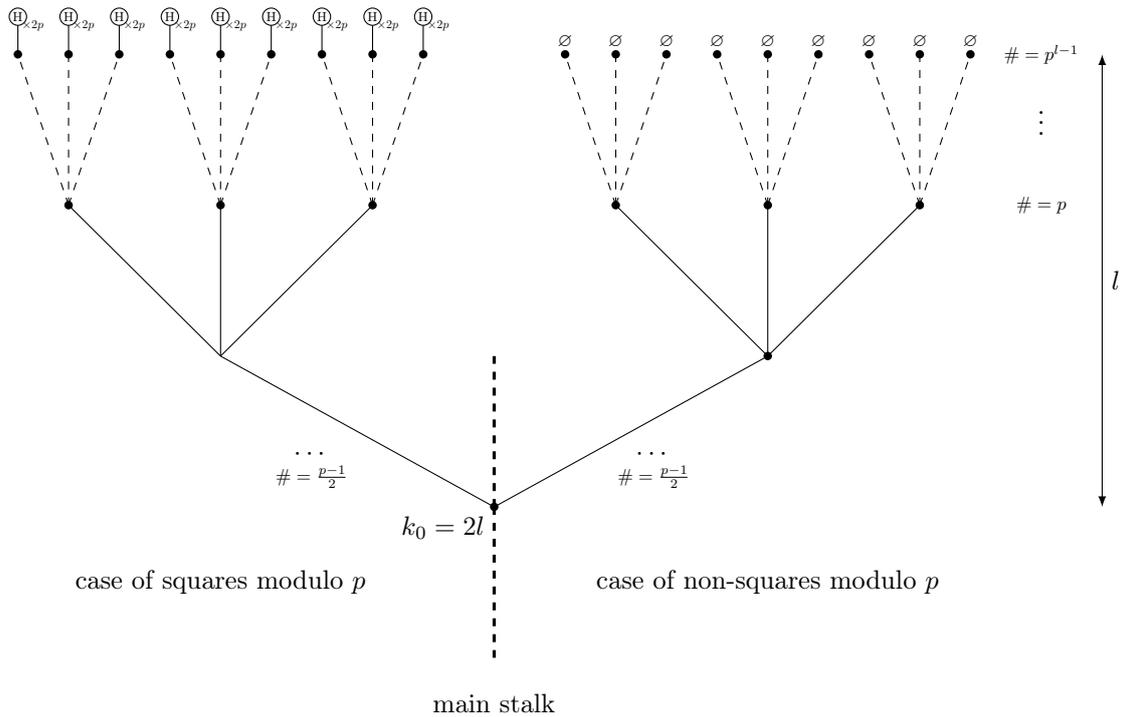

\subsection{The Poincaré series of $x^2-y^3$}

The Poincaré series $S(T)$ of $P(x,y) = x^2-y^3$ for $p>2$ is the sum of the Poincaré series computed for each part of the trunk:
\[
S(T) = S_A(T) + S_B(T) + S_C(T) + S_D(T) + S_E(T) + S_F(T).
\]
After computation, we finally obtain a rather simple fraction:
\[
S(T) = \frac{p^6 + (p^4-p^3)T^2-T^6}{(p-T)(p^5-T^6)}.
\]
This formula is also valid for $p=2$. It is indeed the same formula as the one obtained in \cite{PoVe}.

If we expand this fraction as a power series, we obtain:
\begin{align*}
	S(T) &= 1 + \tfrac{1}{p}T 
	+ \left(\tfrac{2}{p^2}-\tfrac{1}{p^3}\right)T^2
	+ \left(\tfrac{2}{p^3}-\tfrac{1}{p^4}\right)T^3
	+ \left(\tfrac{2}{p^4}-\tfrac{1}{p^5}\right)T^4 \\
	&\quad + \left(\tfrac{1}{p^5}-\tfrac{1}{p^6}\right)T^5
	+ \left(\tfrac{1}{p^5}+\tfrac{1}{p^6}-\tfrac{1}{p^7}\right)T^6
	+ \left(\tfrac{1}{p^6}+\tfrac{1}{p^7}-\tfrac{1}{p^8}\right)T^7
	+ \cdots	
\end{align*}


\subsection{The number of solutions}

By identifying the coefficients of $S(T)$ with those in the definition
$S(T) = \sum_{e\ge0} \frac{N_e}{p^{2e}} T^e$, we obtain the number of solutions of
\[
x^2-y^3 \equiv 0 \pmod{p^e}.
\]

We thus find:
\[
\begin{array}{ll}
	N_1 = p 
	&\qquad N_5 = p^{10}\left(\tfrac{1}{p^5}-\tfrac{1}{p^6}\right) \\
	N_2 = p^4\left(\tfrac{2}{p^2}-\tfrac{1}{p^3}\right) 
	&\qquad N_6 = p^{12}\left(\tfrac{1}{p^5}+\tfrac{1}{p^6}-\tfrac{1}{p^7}\right) \\
	N_3 = p^6\left(\tfrac{2}{p^3}-\tfrac{1}{p^4}\right) 
	&\qquad N_7 = p^{14}\left(\tfrac{1}{p^6}+\tfrac{1}{p^7}-\tfrac{1}{p^8}\right) \\
	N_4 = p^8\left(\tfrac{2}{p^4}-\tfrac{1}{p^5}\right) 
	&\qquad \cdots
\end{array} 
\]	

For example, for $p=5$, we obtain:
\[
\begin{array}{ll}
	N_1 = 5 
	&\qquad N_5 = \num{5625} \\
	N_2 = 45 
	&\qquad N_6 = \num{90625} \\
	N_3 = 225
	&\qquad N_7 = \num{453125} \\
	N_4 = \num{1125}
	&\qquad \cdots
\end{array} 
\]


These are values $N_e$ that can be checked by computer by enumerating all solutions when $e$ is small. For larger $e$, the Poincaré series allows one to compute $N_e$ efficiently. For example, still for $p=5$ and for $e=10$,
among the $p^{2e} = \num{95367431640625}$ possible pairs $(x,y)$, the number of solutions of
\[
x^2-y^3 \equiv 0 \pmod{5^{10}}
\]
is
\[
N_{10} = \num{95703125}.
\]

\end{document}